\pdfoutput=1
\providecommand{\noopsort}[1]{} 

\documentclass{amsart}

\usepackage[mathscr]{eucal}
\usepackage{amssymb}
\usepackage[usenames,dvipsnames]{xcolor} 
\usepackage[normalem]{ulem}
\usepackage{amsthm}
\usepackage{bbold}
\usepackage{comment}
\usepackage{enumitem}
\usepackage{amsmath}
\usepackage{tikz-cd}

\usepackage[unicode]{hyperref} 
\hypersetup{
    colorlinks, linkcolor=Blue,
    citecolor=Blue, urlcolor=Blue,
    breaklinks=true
}
\usepackage[nameinlink,capitalise,noabbrev]{cleveref}

\numberwithin{equation}{section}
\usepackage[all]{xy}
\xyoption{line}
\usepackage{graphicx}
\usepackage{mathtools}

\newtheorem{ThmAlpha}{Theorem}

\swapnumbers 

\newtheorem{Thm}[equation]{Theorem}
\newtheorem*{Thm*}{Theorem}
\newtheorem{Prop}[equation]{Proposition}
\newtheorem{Lem}[equation]{Lemma}
\newtheorem{Cor}[equation]{Corollary}

\theoremstyle{remark}
\newtheorem{Def}[equation]{Definition}
\newtheorem*{Def*}{Definition}
\newtheorem{Ter}[equation]{Terminology}
\newtheorem{Not}[equation]{Notation}
\newtheorem{Exa}[equation]{Example}
\newtheorem{Hyp}[equation]{Hypothesis}
\newtheorem{Rec}[equation]{Recollection}
\newtheorem{Rem}[equation]{Remark}

\tikzset{
    labelrotatebelow/.style={anchor=north, rotate=90, inner sep=1.0mm}
}
\tikzset{
    labelrotateabove/.style={anchor=south, rotate=90, inner sep=1.0mm}
}

\newcommand{\nc}{\newcommand}
\nc{\dmo}{\DeclareMathOperator}

\renewcommand{\emptyset}{\varnothing}
\nc{\Mid}{\,\big|\,}
\nc{\SET}[2]{\big\{\,#1\Mid#2\,\big\}} 

\dmo{\Ker}{Ker}
\dmo{\End}{End}
\dmo{\Hom}{Hom}
\dmo{\im}{im}
\nc{\inv}{^{-1}}

\nc{\hook}{\hookrightarrow}
\nc{\adjto}{\rightleftarrows}
\nc{\xra}{\xrightarrow}

\dmo{\gen}{gen}
\dmo{\cl}{cl}
\dmo{\loc}{loc}
\dmo{\con}{con}
\dmo{\stmod}{stmod}
\dmo{\StMod}{StMod}
\dmo{\Proj}{Proj}
\dmo{\SH}{SH}
\nc{\SHc}{{\SH^c}}
\nc{\SHp}{{\SH_{(p)}}}
\nc{\SHcp}{{\SH^c_{(p)}}}

\dmo{\Ann}{Ann}
\dmo{\WeakAss}{WeakAss}
\dmo{\Ass}{Ass}
\dmo{\modname}{mod}
\dmo{\Mod}{Mod}
\nc{\MMod}[1]{\Mod(#1)}
\nc{\mmod}[1]{\modname(#1)}
\dmo{\Der}{D}
\dmo{\Derqc}{D_{qc}}
\dmo{\Derb}{D^{b}}
\dmo{\Derperf}{D^{perf}}
\nc{\Rder}{\mathrm{R}}
\nc{\Lder}{\mathrm{L}}

\dmo{\cone}{cone}
\dmo{\supp}{supp}
\dmo{\Supp}{Supp}
\nc{\SuppBIK}{\Supp_{\textup{BIK}}}
\nc{\kos}[2]{{#1/\!\!/#2}}
\dmo{\Cosupp}{Cosupp}
\nc{\haux}{\mathrm{h}}
\nc{\Supph}{\Supp^\haux}
\nc{\Spch}{\Spc^\haux}
\dmo{\Spec}{Spec}
\dmo{\Spech}{Spec^h}
\dmo{\Spc}{Spc}

\nc{\ideal}[1]{\langle #1\rangle}
\dmo{\thick}{thick}
\nc{\thickt}[1]{\thick_\otimes\ideal{#1}}
\nc{\Loc}[1]{\operatorname{Loc}\ideal{#1}}
\nc{\Loco}[1]{\operatorname{Loc}_{\otimes}\hspace{-0.3ex}\ideal{#1}}
\nc{\Colocid}[1]{\operatorname{Colocid}\ideal{#1}}

\dmo{\rmH}{H}
\nc{\HFp}{{\rmH \hspace{-0.15em}\bbF_{\hspace{-0.1em}p}}}

\nc{\frakp}{\mathfrak{p}}
\nc{\frakq}{\mathfrak{q}}
\nc{\fraka}{\mathfrak{a}}
\nc{\frakb}{\mathfrak{b}}
\nc{\bbN}{\mathbb{N}}
\nc{\bbZ}{\mathbb{Z}}
\nc{\bbF}{\mathbb{F}}
\nc{\cal}[1]{\mathcal{#1}}
\nc{\cat}[1]{\mathscr{#1}}
\nc{\unit}{\mathbb{1}}

\nc{\ie}{{i.e.}, }
\nc{\cf}{{cf.~}}
\nc{\aka}{{a.\,k.\,a.}\ }
\nc{\eg}{{\sl e.\,g.}}
\nc{\loccit}{{\sl loc.\,cit.}\ }

\newcounter{enum-resume-hack}

\usepackage{todonotes}

\begin{document}

\title[Support theories for non-Noetherian tt-categories]{Support theories for non-Noetherian \\ tensor triangulated categories}

\author{Changhan Zou}
\date{August 7, 2025}

\address{Changhan Zou, Mathematics Department, UC Santa Cruz, 95064 CA, USA}
\email{czou3@ucsc.edu}
\urladdr{\href{https://people.ucsc.edu/~czou3}{https://people.ucsc.edu/$\sim$czou3}}

\begin{abstract}
We extend the support theory of Benson--Iyengar--Krause to the non-Noetherian setting by introducing a new notion of small support for modules. This enables us to prove that the stable module category of a finite group is canonically stratified by the action of the Tate cohomology ring, despite the fact that this ring is rarely Noetherian. In the tensor triangular context, we compare the support theory proposed by W. Sanders (which extends the Balmer--Favi support theory beyond the weakly Noetherian setting) with our generalized BIK support theory. When the Balmer spectrum is homeomorphic to the Zariski spectrum of the endomorphism ring of the unit, the two support theories coincide as do their associated theories of stratification. We also prove a negative result which states that the Balmer--Favi--Sanders support theory can only stratify categories whose spectra are weakly Noetherian. This provides additional justification for the weakly Noetherian hypothesis in the work of Barthel, Heard and B. Sanders. On the other hand, the detection property and the local-to-global principle remain interesting in the general setting.
\end{abstract}

\maketitle

\vspace{-3ex}
{
\hypersetup{linkcolor=black}
\tableofcontents
}
\vspace{-3ex}

\section{Introduction}
The fundamental theorem of tensor triangular geometry \cite{Balmer05a} unifies major classification theorems in algebraic geometry, modular representation theory, and stable homotopy theory \cite{Thomason97,BensonCarlsonRickard97,HopkinsSmith98}. The theorem states that the radical thick ideals of an essentially small tensor triangulated category are classified by the Thomason subsets of its Balmer spectrum via the universal theory of support. Such a category often arises as the subcategory of compact objects inside a bigger rigidly-compactly generated tensor triangulated category. Understanding these big tt-categories leads to the problem of classifying their localizing ideals via some theory of support for big (non-compact) objects. The first such classification theorem was obtained by Neeman \cite{Neeman92a}, who proved that for a commutative Noetherian ring~$A$, the usual cohomological support (defined by tensoring with the residue fields) induces a bijection
\[
\{\text{localizing subcategories of }\Der(A)\} \xra{\sim} \{\text{subsets of }\Spec(A)\}.
\]
Neeman \cite{Neeman00} also showed that such a classification can fail if $A$ is not Noetherian. He exhibited a truncated polynomial ring $A$ in infinitely many variables with the property that $\Der(A)$ has lots of localizing subcategories while $\Spec(A)$ consists of a single point. In fact, Dwyer and Palmieri \cite{DwyerPalmieri08} constructed examples of nontrivial tensor-nilpotent objects in $\Der(A)$. This demonstrates that the cohomological support need not even detect vanishing of objects if the ring is non-Noetherian.

Nevertheless, various authors have constructed support theories for big categories under certain Noetherian hypotheses and used them to prove analogous classification theorems in different subjects. The current paper aims to study these notions of support without making any Noetherian assumptions.
\[ \ast \ast \ast \]

In a series of papers \cite{BensonIyengarKrause08,BensonIyengarKrause11a,BensonIyengarKrause11b} Benson, Iyengar, and Krause developed a theory of support for objects in a compactly generated triangulated category $\cat T$ equipped with an action of a graded-commutative Noetherian graded ring $R$. The BIK support function $\SuppBIK$ assigns to every object $t \in \cat T$ a subset of the homogeneous Zariski spectrum $\Spech(R)$. Suppose that $\cat T$ is a rigidly-compactly generated tensor triangulated category. The BIK support induces a map
\[
\{\text{localizing ideals of }\cat T\} \xra{\SuppBIK} \{\text{subsets of }\SuppBIK(\cat T)\}
\]
where $\SuppBIK(\cat T)\subseteq \Spech(R)$ is the space of supports. The category $\cat T$ is said to be \emph{BIK-stratified by $R$} if this map is a bijection. A major application of this machinery is that the stable module category $\StMod(kG)$ of a finite group~$G$ (or more generally, a finite group scheme) over a field $k$ of positive characteristic is BIK-stratified by the group cohomology ring $H^*(G,k)$ (see \cite{BensonIyengarKrause11a,BensonIyengarKrausePevtsova18}).

Note that this construction of support relies on an auxiliary action by a Noetherian ring. Thanks to the monoidal structure, the graded endomorphism ring~$\End_{\cat T}^*(\unit)$ of the unit object canonically acts on $\cat T$ but it has no reason to be Noetherian in general. For instance, $\End_{\StMod(kG)}^*(\unit)$ is the Tate cohomology ring $\smash{\hat{H}^*(G,k)}$ which is usually non-Noetherian. This motivated the author to develop a BIK-style support theory without assuming that the ring acting on the category is Noetherian. A key ingredient of this theory is a new notion (\cref{def:supp}) of small support for modules over commutative rings that are not necessarily Noetherian. This notion makes use of weakly associated primes which behave particularly well for non-Noetherian rings. For example, we show that torsion modules can still be recognized from their supports; see \cref{lem:supptorsion}.
\[ \ast \ast \ast \]

In the tensor triangular world, Balmer and Favi \cite{BalmerFavi11} proposed a notion of support for rigidly-compactly generated tensor triangulated categories which takes values in the Balmer spectrum $\Spc(\cat T^c)$ of compact objects. The construction of the Balmer--Favi support requires $\Spc(\cat T^c)$ to be \emph{weakly Noetherian}, meaning that for every $\cat P\in\Spc(\cat T^c)$ there exist Thomason subsets $U$ and $V$ with $\{\cat P\}=U\cap V^c$. A point satisfying this condition is called \emph{weakly visible}. Based on this notion of support, Barthel, Heard and B. Sanders \cite{BarthelHeardSander23b} developed a theory of stratification which applies to any category $\cat T$ whose spectrum $\Spc(\cat T^c)$ is weakly Noetherian. In particular, $\cat T$ is said to be \emph{stratified} if the map induced by the Balmer--Favi support
\[
\{\text{localizing ideals of }\cat T\} \xra{\Supp_{\text{BF}}} \{\text{subsets of }\Spc(\cat T^c)\}
\]
is a bijection. Moreover, the Balmer--Favi support theory is the ``universal" one for stratification in the weakly Noetherian context; see \cite[Theorem~7.6]{BarthelHeardSander23b} for a precise statement.

However, there are tensor triangulated categories whose spectra are not weakly Noetherian. A prominent example is the stable homotopy category. Another example is the derived category of a polynomial ring in infinitely many variables. Therefore, a more general notion of support is needed.

W. Sanders \cite{BillySanders17pp} has proposed a generalization of the Balmer--Favi support theory which does not require the spectrum to be weakly Noetherian. We call this support the \emph{tensor triangular support}. A crucial feature of this theory is that the support of an object is a closed subset with respect to a certain topology on the Balmer spectrum called the \emph{localizing topology}. The localizing topology is generated by subsets of the form $U\cap V^c$, where $U$ and $V$ are Thomason subsets. Hence the Balmer spectrum is weakly Noetherian precisely when its localizing topology is discrete. We call subsets closed with respect to the localizing topology \emph{localizing closed}. Therefore, the following definition recovers stratification in the sense of \cite{BarthelHeardSander23b} when $\Spc(\cat T^c)$ is weakly Noetherian:
\begin{Def*}
A rigidly-compactly generated tensor triangulated category $\cat T$ is \emph{stratified} if the map induced by the tensor triangular support
\begin{equation}\label{eq:Supp-strat}
\{\text{localizing ideals of }\cat T\} \xra{\Supp} \{\text{localizing closed subsets of }\Spc(\cat T^c)\}
\end{equation}
is a bijection.
\end{Def*}
At this point, one may wonder if there is any stratified category with non-weakly Noetherian spectrum. Perhaps surprisingly, the answer is no:
\begin{ThmAlpha}[\cref{thm:stratified-weakly-noetherian}]\label{thm:b}
If a rigidly-compactly generated tensor triangulated category $\cat T$ is stratified then the Balmer spectrum $\Spc(\cat T^c)$ is weakly Noetherian.
\end{ThmAlpha}

Nonetheless, without the weakly Noetherian assumption, we can still define a notion of local-to-global principle (\cref{def:ltg}) which is a useful consequence of stratification. In \cite[Theorem~3.21]{BarthelHeardSander23b} it was shown that if $\Spc(\cat T^c)$ is Noetherian then $\cat T$ satisfies the local-to-global principle. We strengthen their result as follows:
\begin{ThmAlpha}[\cref{thm:Hochweakscatter}]\label{thm:c}
A rigidly-compactly generated tensor triangulated category $\cat T$ satisfies the local-to-global principle if the Balmer spectrum $\Spc(\cat T^c)$ is Hochster weakly scattered.
\end{ThmAlpha}
The relation between Noetherian and Hochster weakly scattered spectral spaces can be depicted as follows (see \cref{rem:hoch-scatter}):
\[
\text{Noetherian} \implies \text{Hochster scattered} \iff
\begin{gathered}
\text{Hochster weakly scattered} \\ + \\ \text{weakly Noetherian}.
\end{gathered} 
\]

An immediate consequence of \cref{thm:b} is that the stable homotopy category is not stratified. Nevertheless, we can show that it satisfies the local-to-global principle and hence the detection property. As an application, we show that any spectrum supported on the line at height infinity in the chromatic picture must be dissonant (\cref{cor:dccbik}). Note that the Balmer--Favi support does not ``see" this point since it is not weakly visible.

In order to prove \cref{thm:b}, we study relations between the \emph{homological support} proposed by Balmer \cite{Balmer20_bigsupport} and the tensor triangular support, which were established in \cite{BarthelHeardSander23a} under the weakly Noetherian hypothesis. Our generalizations of the comparison results in \cite{BarthelHeardSander23a} may be of independent interest; see \cref{sec:strat}.
\[ \ast \ast \ast \]

For a rigidly-compactly generated tensor triangulated category $\cat T$, we now have the tensor triangular support $\Supp$ which takes values in the Balmer spectrum $\Spc(\cat T^c)$ and the canonical BIK support $\SuppBIK$ which takes values in the homogeneous Zariski spectrum $\Spech(\End_{\cat T}^*(\unit))$. Moreover, there is a comparison map
\[
\rho \colon \Spc(\cat T^c)\to\Spech(\End_{\cat T}^*(\unit))
\]
introduced in \cite{Balmer10a}. It is then natural to ask how the two support theories are related. We prove the following:
\begin{ThmAlpha}[\cref{thm:bikcomparison}]\label{thm:d}
Let $\cat T$ be a rigidly-compactly generated tensor triangulated category such that $\rho$ is a homeomorphism. Then $\rho(\Supp(t))=\SuppBIK(t)$ for any $t\in\cat T$.
\end{ThmAlpha}
Inspired by \eqref{eq:Supp-strat}, we say that $\cat T$ is \emph{cohomologically stratified} if the map induced by the canonical BIK support
\[
\{\text{localizing ideals of }\cat T\} \xra{\SuppBIK} \{\text{closed subsets of }\SuppBIK(\cat T)\}
\]
is a bijection, where $\SuppBIK(\cat T)$ inherits the localizing topology on $\Spech(\End^*_{\cat T}(\unit))$. A corollary of \cref{thm:d} is that if the comparison map $\rho$ is a homeomorphism then $\cat T$ is stratified if and only if it is cohomologically stratified; see \cref{cor:coh-strat}. However, a category can be cohomologically stratified without $\rho$ being a homeomorphism. This is the case for the following example:
\begin{ThmAlpha}[\cref{thm:stmod}]\label{thm:e}
The stable module category $\StMod(kG)$ is cohomologically stratified.
\end{ThmAlpha}
\[ \ast \ast \ast \]

The paper is organized as follows. In \cref{sec:pre} we record some basic facts about spectral spaces. In particular, we discuss the localizing topology (and its relation with the constructible topology), which will be used throughout this paper. In \cref{sec:supp-module} we define the notion of (small) support for modules over (non-Noetherian) commutative rings. In \cref{sec:bik} we establish the non-Noetherian BIK support theory. In \cref{sec:supp} we study how the tensor triangular support behaves under geometric functors. In \cref{sec:detection} we prove that the detection property for the tensor triangular support is an algebraically local property. In \cref{sec:ltg} we introduce a local-to-global principle and prove \cref{thm:c}. In \cref{sec:strat} we establish \cref{thm:b}. Finally, we prove \cref{thm:d} and \cref{thm:e} in \cref{sec:comparison}.

\subsection*{Acknowledgements}
The author is grateful to Beren Sanders for inspiring discussions and his constant support. He thanks Paul Balmer for useful conversations and the organizers of the Oberwolfach workshop \emph{Tensor-Triangular Geometry and Interactions} for their invitation to present some part of this work.

\medskip
\section{Preliminaries on spectral spaces}\label{sec:pre}
We start by recalling some basic concepts concerning spectral spaces.

\begin{Def}
Let $X$ be a spectral space in the sense of \cite{DickmannSchwartzTressl19}. We call a subset of $X$ \emph{Thomason} if it is a union of closed subsets, each of which has quasi-compact complement. The Thomason subsets form the open subsets of a dual spectral topology on $X$ called the \emph{Hochster dual topology}.\footnote{The Hochster dual topology is called the \emph{inverse topology} in \cite{DickmannSchwartzTressl19}.} We write $X^*$ for $X$ equipped with the Hochster dual topology.
\end{Def}

\begin{Def}\label{def:weaklyvisible}
Let $X$ be a spectral space. A subset $W$ of $X$ is said to be \emph{weakly visible} if there exist Thomason subsets $U$ and $V$ such that $W=U\cap V^c$. In particular, we say a point~$x\in X$ is weakly visible if the singleton $\{x\}$ is weakly visible. The spectral space $X$ is said to be \emph{weakly Noetherian} if every point of $X$ is weakly visible.
\end{Def}

\begin{Exa}\label{exa:weaklynoe}
Every Noetherian spectral space and every profinite space is weakly Noetherian; see \cite[Remarks~2.2 and 2.4]{BarthelHeardSander23b}.
\end{Exa}

\begin{Rem}\label{rem:Thomason-closed}
Let $X$ be a spectral space. A Thomason subset of $X$ is a union of Thomason closed subsets of $X$. Indeed, this follows from \cite[Lemma~3.3]{Sanders13} which states that the Thomason closed subsets are precisely the closed subsets with quasi-compact complements. It also follows that the closure $\overline{\{y\}}$ of a point $y\in X$ is the intersection of all Thomason closed subsets containing $y$, since the quasi-compact open subsets form a basis for the spectral topology. Therefore
\[
\gen(x) \coloneqq \SET{y\in X}{x\in \overline{\{y\}}}
\]
is the complement of the largest Thomason subset not containing $x$. Given the discussion above, we see that a point $x\in X$ is weakly visible if and only if
\[
\{x\}=Z\cap\gen(x)
\]
for some Thomason closed subset $Z$.
\end{Rem}

\begin{Rem}
The notion of a weakly visible subset leads to the following definition introduced by W. Sanders \cite{BillySanders17pp}:
\end{Rem}

\begin{Def}\label{def:loc-top}
Let $X$ be a spectral space. The weakly visible subsets of $X$ form a basis of open subsets for a topology on $X$ called the \emph{localizing topology} of $X$. We write $X_{\loc}$ for $X$ equipped with the localizing topology, and for any subset $S$ of $X$ we write $\overline{S}^{\loc}$ for the closure of $S$ in $X_{\loc}$.
\end{Def}

\begin{Rem}\label{rem:loc-top}
Note that a spectral space $X$ is weakly Noetherian if and only if its localizing topology is discrete. We will call a subset of $X$ \emph{localizing closed} if it is closed with respect to the localizing topology. Thus, $X$ is weakly Noetherian if and only if every subset of $X$ is localizing closed.
\end{Rem}

\begin{Rem}\label{rem:loc-top-con-top}
Recall that the \emph{constructible topology} on a spectral space $X$ is the topology generated by the sets $U\cap V^c$ with $U$ and $V$ Thomason \emph{closed} subsets of $X$. We write $X_{\con}$ for $X$ equipped with the constructible topology. From the definitions we see that the localizing topology is finer than the constructible topology. Note that the constructible topology is discrete if and only if the spectral space is finite; see \cite[Example~1.3.12 and Theorem~1.3.14]{DickmannSchwartzTressl19}. Therefore, any infinite weakly Noetherian spectral space provides an example whose localizing topology is strictly finer than the constructible topology. An explicit example is:
\end{Rem}

\begin{Exa}
Let $S^+$ denote the one-point compactification of a discrete infinite space $S$. We denote the point at infinity by $\infty$. Note that the space $S^+$ is a profinite space. We now define a partial order on $S^+$ by
\[
s\le t \iff s=\infty \text{ or } s=t\in S.
\]
This is a spectral order and hence yields a Priestly space whose associated spectral space is denoted by $S_{\infty}$; see \cite[1.6.13]{DickmannSchwartzTressl19} for details. Note that the constructible topology on~$S_{\infty}$ coincides with the original topology on $S^+$. However, by \cite[1.6.15(iv) and (v)]{DickmannSchwartzTressl19} the localizing topology on~$S_{\infty}$ is discrete and is therefore strictly finer than the constructible topology on $S_{\infty}$.
\end{Exa}

\begin{Rem}
On the other hand, there are examples where the localizing and constructible topologies coincide:
\end{Rem}

\begin{Exa}
Consider the space $\bbN\cup\{\infty\}$ of extended natural numbers whose nonempty open subsets are of the form $[n,\infty]$ for $n\in\bbN$. This is a spectral space and we denote its Hochster dual by $X$. The space $X$ coincides with the Balmer spectrum $\Spc(\SHcp)$ of the $p$-local stable homotopy category $\SHp$; see \cite{HopkinsSmith98} and \cite[Corollary~9.5]{Balmer10a}. Since every Thomason subset of $X$ is closed, the localizing topology coincides with the constructible topology.
\end{Exa}

\begin{Rem}
In general, the localizing topology is not a spectral topology. In fact, the localizing topology is quasi-compact if and only if it coincides with the constructible topology. Indeed, this follows from the fact that a continuous surjection from a quasi-compact space to a Hausdorff space is a topological quotient, in light of the continuous bijection $X_{\loc}\to X_{\con}$.
\end{Rem}

\begin{Rem}
To conclude this section, we introduce a class of spectral spaces whose definition is somewhat technical but it turns out that some results which hold for Noetherian spectral spaces extend to this class of spaces; see \cref{thm:Hochweakscatter}.
\end{Rem}

\begin{Def}
Let $S$ be a subset of a topological space $X$. A point $x$ is an \emph{isolated point} of $S$ if there exists an open subset $U$ of $X$ such that $\{x\}=U\cap S$. More generally, a point $x$ is a \emph{weakly isolated point} of $S$ if there exists an open subset $U$ of $X$ such that $\{x\}\subseteq U\cap S\subseteq \overline{\{x\}}$. A topological space $X$ is said to be \emph{(weakly) scattered} if every nonempty closed subset of $X$ has a (weakly) isolated point. See \cite{NiefieldRosenthal87} for further discussion.
\end{Def}

\begin{Def}\label{def:hoch-scatter}
A spectral space $X$ is \emph{Hochster (weakly) scattered} if its Hochster dual $X^*$ is (weakly) scattered.
\end{Def}

\begin{Rem}\label{rem:hoch-scatter}
A spectral space is Hochster scattered if and only if it is weakly Noetherian and Hochster weakly scattered; see \cite[Lemma~7.16]{BillySanders17pp}. All Noetherian spectral spaces are Hochster scattered; see \cite[Lemma~7.17(1)]{BillySanders17pp}, for example.
\end{Rem}

\begin{Exa}
Let $R$ be a non-Noetherian absolutely flat commutative ring. The Zariski spectrum $\Spec(R)$ is not Noetherian by \cite[Lemma~3.6]{Stevenson14a}. Nevertheless, if $R$ is semi-artinian then $\Spec(R)$ is Hochster scattered. Indeed, since $\Spec(R)$ carries the constructible topology (that is, $\Spec(R)=\Spec(R)_{\con}$) and has Cantor-Bendixson rank (by the proof of \cite[Theorem~6.4]{Stevenson17}), it then follows from \cite[Lemma~7.17(2)]{BillySanders17pp} that $\Spec(R)$ is Hochster scattered.
\end{Exa}

\section{Small support for modules}\label{sec:supp-module}
We now introduce a notion of small support for graded modules over graded-commutative graded rings that are not necessarily Noetherian. Our definition uses weakly associated primes, which behave better than associated primes in the absence of the Noetherian assumption.

\begin{Not}
For this section, $R$ will denote a $\bbZ$-graded graded-commutative ring. Ideals and modules will always be graded. The abelian category of $R$-modules and degree-zero homomorphisms will be denoted $\MMod{R}$. We write $\Spech(R)$ for the homogeneous Zariski spectrum of $R$. It is a spectral space. For any subset $S$ of $\Spech(R)$, we write $\cl(S)$ for the specialization closure of $S$. Given any ideal $\fraka$ of~$R$ and any prime ideal $\frakp$ in $\Spech(R)$, we write $\cal V(\fraka)$ for the set of prime ideals containing $\fraka$ and $\gen(\frakp)$ for the generalization closure of~$\frakp$. The complement of $\gen(\frakp)$ is denoted by $\cal Z(\frakp)$. We refer the reader to \cite[Section~1.5]{BrunsHerzog98} and \cite[Section~2]{DellAmbrogioStevenson13} for more on graded commutative algebra.
\end{Not}

\begin{Def}
The \emph{big support} of an $R$-module $M$ is defined as
\[
\Supp_R M\coloneqq\SET{\frakp\in\Spech(R)}{M_{\frakp}\neq0}
\]
where $M_{\frakp}$ is the graded localization of~$M$ at $\frakp$. 
\end{Def}

\begin{Def}
Let $M$ be an $R$-module. A prime $\frakp\in\Spech(R)$ is said to be \emph{associated} to~$M$ if there exists a homogeneous element~$m$ in $M$ such that $\frakp=\Ann(m)$, the annihilator of~$m$. We denote the set of associated primes of $M$ by $\Ass(M)$. More generally, a prime $\frakp$ is said to be \emph{weakly associated} to $M$ if there exists a homogeneous element $m$ in $M$ such that $\frakp$ is minimal among the primes containing $\Ann(m)$. The set of weakly associated primes of $M$ is denoted by $\WeakAss(M)$.
\end{Def}

\begin{Lem}\label{lem:weakass}
Let $M$ be an $R$-module and $\frakp\in\Spech(R)$. The following hold:
\begin{enumerate}
\item $\Ass(M)\subseteq\WeakAss(M)\subseteq\Supp_R M$. If $R$ is Noetherian then we have $\Ass(M)=\WeakAss(M)$.
\item $M=0$ if and only if $\WeakAss(M)=\emptyset$.
\item $\WeakAss(M_{\frakp})=\WeakAss(M)\cap\gen(\frakp)$.
\item $\cl(\WeakAss(M))=\Supp_R M$.
\end{enumerate}
\end{Lem}

\begin{proof}
For parts (a), (b), and (c), the proofs in \cite[\href{https://stacks.math.columbia.edu/tag/0589}{Lemma 0589}, \href{https://stacks.math.columbia.edu/tag/058A}{Lemma 058A}, \href{https://stacks.math.columbia.edu/tag/0588}{Lemma 0588}, \href{https://stacks.math.columbia.edu/tag/05C9}{Lemma 05C9}]{stacks-project} carry over to the graded setting. Part (d) is a direct consequence of (b) and (c) since $\Supp_R M$ is always specialization closed. See also \cite[IV, Exercise~17, page 289]{Bourbaki_CommAlg1-7}.
\end{proof}

\begin{Rem}
When $R$ is not Noetherian, there can exist a nonzero $R$-module $M$ such that $\Ass(M)=\emptyset$; see \cite[\href{https://stacks.math.columbia.edu/tag/05BX}{Remark 05BX}]{stacks-project}, for example. We now define the (small) support for modules:
\end{Rem}

\begin{Def}\label{def:supp}
The \emph{support} of an $R$-module $M$ is the set
\[
\supp_R M\coloneqq\WeakAss(M).
\]
\end{Def}

\begin{Rem}
From \cref{lem:weakass}(b) we see that the support detects vanishing of a module:
\begin{equation}\label{eq:suppdetection}
\supp_R M=\emptyset \iff M=0.
\end{equation}
\end{Rem}

\begin{Rem}\label{rem:suppdiffer}
\cref{def:supp} is different from the small support defined in \cite{BensonIyengarKrause08} when $R$ is Noetherian. Assume that $R$ is Noetherian. We have $\supp_R M=\Ass(M)$ by \cref{lem:weakass}(a). In \cite[Section~2]{BensonIyengarKrause08} the authors defined the small support of a module $M$ by choosing a minimal injective resolution of $M$ and then invoking the structure theorem for injective modules over Noetherian rings. More precisely, their small support of $M$ is defined as
\[
\left\{\frakp\in\Spech(R)\left|
{\begin{gathered}
\exists i \ge 0 \colon I^i \text{ has a direct summand isomorphic to} \\
\text{a shifted copy of the injective hull $E(R/\frakp)$}
\end{gathered}}\right.\right\}
\]
where $I^*$ is any minimal injective resolution of $M$. By \cite[Theorem~3.6.3]{BrunsHerzog98} this is equal to $\bigcup_{i \ge 0}\Ass(I^i)$, which contains $\Ass(M)$ since $M$ is a submodule of $I^0$. Nevertheless, our notion of support still satisfies some of the properties derived in \cite[Section~2]{BensonIyengarKrause08}:
\end{Rem}

\begin{Lem}\label{lem:suppproperties}
Let $M$ be an $R$-module and $\cal V \subseteq \Spech(R)$ a specialization closed subset. The following hold:
\begin{enumerate}
\item $\supp_R M \subseteq \cl(\supp_R M)=\Supp_R M$.
\item $\supp_R M_{\frakp}=\supp_R M\cap\gen(\frakp)$.
\item $\supp_R M\subseteq\cal V\iff M_{\frakp}=0 \text{ for all }\frakp\notin\cal V$.
\end{enumerate}
\end{Lem}

\begin{proof}
The equality in (a) is just \cref{lem:weakass}(d) and the inclusion is clear. Part (b) is \cref{lem:weakass}(c). Part (c) follows immediately from (a).
\end{proof}

\begin{Def}
Let $\cal V$ be a specialization closed subset of $\Spech(R)$. Define the full subcategory
\[
\MMod{R}_{\cal V}\coloneqq\SET{M\in\MMod{R}}{\supp_R M\subseteq\cal V}.
\]
\end{Def}

\begin{Lem}\label{lem:modrv}
If $\cal V \subseteq \Spech(R)$ is specialization closed then $\MMod{R}_{\cal V}$ is a localizing Serre subcategory of $\MMod{R}$. That is, $\MMod{R}_{\cal V}$ is closed under coproducts, and for any exact sequence $0\to M'\to M\to M''\to 0$ of $R$-modules, $M$ is in $\MMod{R}_{\cal V}$ if and only if $M'$ and $M''$ are in $\MMod{R}_{\cal V}$.
\end{Lem}

\begin{proof}
This follows from \cref{lem:suppproperties}(c) and the fact that the localization functor at any prime $\frakp$ is exact and preserves coproducts.
\end{proof}

\begin{Def}
Let $M$ be an $R$-module and $\fraka$ an ideal of $R$. The module $M$ is said to be \emph{$\fraka$-torsion} if every element of $M$ is annihilated by a power of $\fraka$.
\end{Def}

\begin{Lem}\label{lem:supptorsion}
Let $M$ be an $R$-module and $\fraka$ an ideal of $R$. We have:
\begin{enumerate}
\item If $M$ is $\fraka$-torsion then $\supp_R M\subseteq\cal V(\fraka)$.
\item $\supp_R M\subseteq\cal V(\fraka)$ if and only if $M$ is $\frakb$-torsion for any finitely generated ideal $\frakb\subseteq \fraka$.
\end{enumerate}
\end{Lem}

\begin{proof}
For part (a), if $M$ is $\fraka$-torsion then we have $M_{\frakp}=0$ for every $\frakp\notin\cal V(\fraka)$ and thus $\supp_R M\subseteq\Supp_R M\subseteq \cal V(\fraka)$. For part (b), recall that by definition $\supp_R M=\WeakAss(M)$. Hence if $\supp_R M\subseteq\cal V(\fraka)$ then
\[
\fraka \subseteq \bigcap_{\frakp\in\WeakAss(M)}\frakp.
\]
Note that for any homogeneous element $m$ of $M$, we have
\[
\sqrt{\Ann(m)} = \bigcap_{\frakp \text{ minimal over }\Ann(m)} \frakp.
\]
By definition, if $\frakp$ is minimal over $\Ann(m)$ then $\frakp \in \WeakAss(M)$. Therefore
\[
\fraka \subseteq \bigcap_{\frakp\in\WeakAss(M)}\frakp \subseteq \bigcap_{\frakp \text{ minimal over }\Ann(m)} \frakp =\sqrt{\Ann(m)}.
\]
It follows that $M$ is $\frakb$-torsion for any finitely generated ideal $\frakb\subseteq\fraka$. The other direction follows from part (a).
\end{proof}

\begin{Rem}
We see from the lemma above that for a finitely generated ideal $\fraka$, an $R$-module $M$ is $\fraka$-torsion if and only if $\supp_R M \subseteq \cal V(\fraka)$. However, this does not always hold when $\fraka$ is not finitely generated. Indeed, in \cite{Rohrer19} the author studied two torsion functors for an ideal $\fraka$ of a commutative ring $R$ (in the ungraded setting): the \emph{small $\fraka$-torsion functor} $\varGamma_{\fraka}$ and the \emph{large $\fraka$-torsion functor} $\overline{\varGamma}_{\fraka}$, which are defined as
\[
\varGamma_{\fraka} M \coloneqq \SET{m\in M}{\fraka^n \subseteq \Ann(m) \text{ for some }n\in \bbN}
\]
and
\[
\overline{\varGamma}_{\fraka} M \coloneqq \SET{m\in M}{\fraka \subseteq \sqrt{\Ann(m)}}.
\]
Hence $M = \varGamma_{\fraka} M$ if and only if $M$ is $\fraka$-torsion. On the other hand, $M =  \overline{\varGamma}_{\fraka} M$ if and only if $\supp_R M \subseteq \cal V(\fraka)$; this follows from \cite[(3.3)(B)]{Rohrer19} and \cref{lem:suppproperties}(a). It is clear that $\varGamma_{\fraka}$ is a subfunctor of $\overline{\varGamma}_{\fraka}$ and that $\varGamma_{\fraka}=\overline{\varGamma}_{\fraka}$ if $\fraka$ is finitely generated. However, these two functors do not coincide in general; see \cite[Section~4]{Rohrer19}.
\end{Rem}

\section{Non-Noetherian BIK support}\label{sec:bik}
Benson, Iyengar, and Krause \cite{BensonIyengarKrause08} developed a theory of support for any compactly generated triangulated category equipped with a central action by a Noetherian graded-commutative graded ring. In this section we show how the Noetherian hypothesis can be removed.

\begin{Ter}
For the rest of the section, we fix a \emph{compactly generated $R$-linear triangulated category} $\cat T$. That is, a compactly generated triangulated category $\cat T$ equipped with a homomorphism of graded rings $R\to Z(\cat T)$ where $Z(\cat T)$ is the graded center of $\cat T$. The full subcategory of compact objects in $\cat T$ is denoted by $\cat T^c$. Given any two objects $x$ and $t$ in $\cat T$, the graded abelian group
\[
\Hom_{\cat T}^*(x,t) \coloneqq \coprod_{n\in\bbZ}\Hom_{\cat T}(x,\Sigma^n t)
\]
has a graded $Z(\cat T)$-module structure and hence is a graded module over $R$; see \cite[Section~4]{BensonIyengarKrause08} for further details.
\end{Ter}

\begin{Rem}
Recall that a localizing subcategory $\cat L$ of $\cat T$ is \emph{strictly localizing} if the inclusion $\cat L\hook \cat T$ admits a right adjoint. This is equivalent to $\cat L$ being the kernel of a Bousfield localization on $\cat T$; see \cite[Proposition~9.1.8]{Neeman01}, for example.
\end{Rem}

\begin{Lem}
Let $\cal V\subseteq\Spech(R)$ be specialization closed. The subcategory
\[
\cat T_{\cal V}\coloneqq\SET{t\in\cat T}{\supp_R \Hom_{\cat T}^*(x,t) \subseteq \cal V\text{ for each }x\in \cat T^c}\] 
is strictly localizing.
\end{Lem}

\begin{proof}
Note that $\cat T_{\cal V}$ is a localizing subcategory of $\cat T$ by \cref{lem:modrv}. The proof of \cite[Proposition~4.5]{BensonIyengarKrause08} then carries over verbatim, thanks to \cref{lem:suppproperties}(c).
\end{proof}

\begin{Rem}\label{rem:bik-exact-triangle}
For an object $t\in \cat T$ and a specialization closed subset $\cal V\subseteq\Spech(R)$, there exists (by the lemma above) a localization triangle
\[
\varGamma_{\cal V}t \to t \to L_{\cal V}t
\]
where $L_{\cal V}$ and $\varGamma_{\cal V}$ are the corresponding Bousfield localization and colocalization functor, respectively. We think of $\varGamma_{\cal V}t$ as the part of $t$ supported on $\cal V$ and $L_{\cal V}t$ as the part of $t$ supported away from $\cal V$. Recall from \cref{rem:suppdiffer} that our $\supp_R$ differs from the small support defined in \cite{BensonIyengarKrause08} when $R$ is Noetherian. Nevertheless, they have the same specialization closure, namely $\Supp_R$; see \cref{lem:suppproperties}(a) and \cite[Lemma~2.2(1)]{BensonIyengarKrause08}. Thus the category $\cat T_{\cal V}$ defined above is the same as the one defined in \cite[Lemma~4.3]{BensonIyengarKrause08}.
Therefore, the localization functor $L_{\cal V}$ is the same as the one in \cite[Definition~4.6]{BensonIyengarKrause08}. In particular, these localization functors satisfy the following composition rules:
\end{Rem}

\begin{Lem}\label{lem:localcomp}
Let $\cal V$ and $\cal W$ be specialization closed subsets of $\Spech(R)$. The following hold:
\begin{enumerate}
\item $\varGamma_{\cal V}\varGamma_{\cal W} \cong \varGamma_{\cal V \cap \cal W} \cong \varGamma_{\cal W}\varGamma_{\cal V}$.
\item $L_{\cal V}L_{\cal W} \cong L_{\cal V \cup \cal W} \cong L_{\cal W}L_{\cal V}$.
\item $\varGamma_{\cal V}L_{\cal W} \cong L_{\cal W}\varGamma_{\cal V}$.
\end{enumerate}
\end{Lem}

\begin{proof}
See \cite[Proposition~6.1]{BensonIyengarKrause08}.
\end{proof}

\begin{Rem}
In \cref{rem:bik-exact-triangle} we have noted that the construction of the localization functors in \cite[Definition~4.6]{BensonIyengarKrause08} depends only on the notion of big support of modules, which does not require the ring $R$ to be Noetherian. However, to show that such a localization functor is a finite localization we want to have a more concrete description of the category $\cat T_{\cal V}$. For example, the objects in $\cat T_{\cal V(\fraka)}$ for a finitely generated ideal~$\fraka$ should be those $t\in\cat T$ with the property that $\Hom_{\cat T}^*(x,t)$ is $\fraka$-torsion for every $x\in\cat T^c$. This was established in \cite[Lemma~2.4(2)]{BensonIyengarKrause08} under the Noetherian hypothesis on the ring $R$. Thanks to \cref{lem:supptorsion}, this remains true for general commutative rings. Our next goal is to show that for any finitely generated ideal~$\fraka$ of $R$ and any prime ideal $\frakp\in\Spech(R)$ the localization functors corresponding to~$\cal V(\fraka)$ and $\cal Z(\frakp)$ are finite localizations, that is, the localizing subcategories $\cat T_{\cal V(\fraka)}$ and $\cat T_{\cal Z(\frakp)}$ are generated by compact objects in $\cat T$. Let us first recall the notion of Koszul objects.
\end{Rem}

\begin{Def}
Let $r\in R$ be a homogeneous element of degree $d$ and let $t$ be an object of $\cat T$. We denote by $\kos t{r}$ any object that fits into an exact triangle
\[
t \xra{r} \Sigma^d t \to \kos t{r}.
\]
This is called a \emph{Koszul object of $r$ on $t$}. Given a finite sequence $\underline{r}=(r_1,\ldots,r_n)$ of homogeneous elements, a \emph{Koszul object of $\underline{r}$ on $t$} is defined iteratively and denoted by $\kos t{\underline{r}}$. For a finitely generated ideal $\fraka$ of $R$, we write $\kos t{\fraka}$ for any Koszul object of any finite sequence of homogeneous generators for $\fraka$; see \cite[Definition~5.10]{BensonIyengarKrause08} for further discussion. A Koszul object $\kos t{\fraka}$ depends on the choice of generating sequence for the ideal $\fraka$. Nevertheless, the thick subcategory generated by $\kos t{\fraka}$ depends only on the radical of $\fraka$ by \cite[Lemma~3.4(2)]{BensonIyengarKrause11a}. Note also that $\kos t{\fraka}$ is compact if $t$ is compact.
\end{Def}

\begin{Lem}\label{lem:suppkos}
For every object $t \in \cat T$ and every finitely generated ideal $\fraka$ of $R$ we have $\kos t{\fraka}\in\cat T_{\cal V(\fraka)}$.
\end{Lem}

\begin{proof}
By \cite[Lemma~5.11(1)]{BensonIyengarKrause08} the $R$-module $\Hom_{\cat T}^*(x,\kos t{\fraka})$ is $\fraka$-torsion for every~$x\in\cat T^c$. \cref{lem:supptorsion} then yields the desired result.
\end{proof}

\begin{Prop}\label{prop:bikalgebraic}
For any finitely generated ideal $\fraka$ of $R$ and any $\frakp\in \Spech(R)$, the categories $\cat T_{\cal V(\fraka)}$ and $\cat T_{\cal Z(\frakp)}$ are compactly generated subcategories of $\cat T$:
\begin{enumerate}
\item $\cat T_{\cal V(\fraka)}=\Loc{\kos x{\fraka}\mid x\in\cat T^c}$.
\item $\cat T_{\cal Z(\frakp)}=\Loc{\kos x{r}\mid x\in\cat T^c \text{ and } r\in R\setminus \frakp\text{ homogeneous}}.$
\end{enumerate}
\end{Prop}

\begin{proof}
By \cref{lem:suppkos} we have $\cat S\coloneqq\Loc{\kos x{\fraka}\mid x\in\cat T^c}\subseteq\cat T_{\cal V(\fraka)}$. Since $\cat S$ is a compactly generated subcategory of $\cat T_{\cal V(\fraka)}$, it is strictly localizing \cite[Proposition~9.1.19]{Neeman01}. Denoting the corresponding colocalization functor by $\varGamma$, we have for any $t\in\cat T_{\cal V(\fraka)}$ an exact triangle
\[
\varGamma t \to t \to s
\]
for some $s\in \cat S^{\perp}$. It remains to prove $s=0$. Let $r_1,\ldots,r_n$ be a sequence of homogeneous generators for $\fraka$. Note that for any $x\in\cat T^c$ we have
\[
\Hom_{\cat T}^*(\kos x{(r_1,\ldots,r_n)},s)=0.
\]
We claim
\[
\Hom_{\cat T}^*(\kos x{(r_1,\ldots,r_{n-1})},s)=0.
\]
Let $m\in\Hom_{\cat T}^*(\kos x{(r_1,\ldots,r_{n-1})},s)$. Since $s\in\cat T_{\cal V(\fraka)}=\bigcap_{i=1}^{n}\cat T_{\cal V(r_i)}\subseteq\cat T_{\cal V(r_n)}$, there exists a positive integer $k$ with $r_n^k m=0$ by \cref{lem:supptorsion}. Let $d$ be the degree of $r_n$. Applying $\Hom_{\cat T}^*(-,s)$ to the exact triangle
\[
\kos x{(r_1,\ldots,r_{n-1})} \\
\xra{r_n}\Sigma^d \kos x{(r_1,\ldots,r_{n-1})}\to \kos x{(r_1,\ldots,r_n)}
\]
yields an exact sequence
\begin{multline*}
0=\Hom_{\cat T}^*(\kos x{\fraka},s) \to \Hom_{\cat T}^*(\kos x{(r_1,\ldots,r_{n-1})},s)[-d] \\
\xra{r_n} \Hom_{\cat T}^*(\kos x{(r_1,\ldots,r_{n-1})},s) \to \Hom_{\cat T}^*(\kos x{\fraka},s)[-1]=0.
\end{multline*}
Thus multiplying by $r_n$ is an isomorphism, so $m=0$. The claim follows. An induction yields $\Hom_{\cat T}^*(x,s)=0$. This is true for all $x\in\cat T^c$. Therefore $s=0$, which establishes (a).

For part (b), set $\cat S\coloneqq\Loc{\kos x{r}\mid x\in\cat T^c\text{ and } r\in R\setminus \frakp\text{ is homogeneous}} \subseteq \cat T_{\cal Z(\frakp)}$. Similarly, for any $t\in\cat T_{\cal Z(\frakp)}$ we have an exact triangle
\[
\varGamma t \to t \to s
\]
with $s\in \cat S^{\perp}$, where $\varGamma$ is the corresponding colocalization functor. It remains to show that $s=0$. Since $s\in\cat T_{\cal Z(\frakp)}$, we have $\supp_R\Hom_{\cat T}^*(x,s)\subseteq\cal Z(\frakp)$ and hence $\Hom_{\cat T}^*(x,s)_{\frakp}=0$ by \cref{lem:suppproperties}(b) and \eqref{eq:suppdetection}, for every $x\in\cat T^c$. It follows that for any homogeneous element $m\in\Hom_{\cat T}^*(x,s)$ there exists some homogeneous element~$r\notin\frakp$ with $rm=0$. Let $d$ be the degree of $r$. The exact sequence
\[
0=\Hom_{\cat T}^*(\kos x{r},s)[-1] \to \Hom_{\cat T}^*(x,s)[-d] \xra{r} \Hom_{\cat T}^*(x,s) \to \Hom_{\cat T}^*(\kos x{r},s)=0
\]
implies that $m=0$. Therefore $s=0$, which completes the proof.
\end{proof}

\begin{Rem}
Now we are ready to define the BIK support for objects in $\cat T$.
\end{Rem}

\begin{Def}\label{def:SuppBIK}
The \emph{BIK support} of an object $t$ in $\cat T$ is defined as
\[
\SuppBIK(t)\coloneqq\left\{\frakp\in\Spech(R)\left|
{\begin{gathered}
\varGamma_{\cal V(\fraka)}L_{\cal Z(\frakp)}t\neq0\text{ for any finitely} \\
\text{ generated ideal }\fraka\text{ contained in }\frakp
\end{gathered}}\right.\right\}.
\]
\end{Def}

\begin{Rem}
If $\frakp\in\Spech(R)$ is finitely generated then \cref{lem:localcomp}(a) implies that $\SuppBIK(t) = \SET{ \frakp \in \Spech(R)}{\varGamma_{\cal V(\frakp)} L_{\cal Z(\frakp)} t \neq 0}$. Therefore the support $\SuppBIK$ recovers the one in \cite{BensonIyengarKrause08} when $R$ is Noetherian.
\end{Rem}

\begin{Rem}
We give the following description of the BIK support which is more flexible.
\end{Rem}

\begin{Prop}\label{prop:suppbikdef}
We have
\[
\SuppBIK(t) = \left\{\frakp\in\Spech(R) \left|
{\begin{gathered}
\varGamma_{\cal V}L_{\cal Z}t\neq0\text{ for any Thomason} \\
\text{ subsets }\cal V,\cal Z\text{ such that }\frakp\in\cal V\cap\cal Z^c
\end{gathered}}\right.\right\}
\]
for any $t\in\cat T$.
\end{Prop}

\begin{proof}
Recall from \cref{rem:Thomason-closed} that a Thomason subset of a spectral space is a union of Thomason closed subsets and that $\cal Z(\frakp)=\gen(\frakp)^c$ is the largest Thomason subset not containing $\frakp$. Moreover, the Thomason closed subsets of $\Spech(R)$ are exactly subsets of the form $\cal V(\fraka)$ for some finitely generated ideal $\fraka$. Therefore the result follows from \cref{lem:localcomp}.
\end{proof}

\begin{Not}
For $S \subseteq \Spech(R)$ we write $\min S$ for the set of prime~ideals in $S$ that are minimal (with respect to inclusion) among the prime ideals in $S$.
\end{Not}

\begin{Thm}\label{thm:suppbik}
For any $t\in\cat T$ we have
\[
\bigcup_{x\in\cat T^c} \min \supp_R\Hom_{\cat T}^*(x,t) \subseteq \SuppBIK(t) \subseteq \bigcup_{x\in \cat G} \Supp_R \Hom_{\cat T}^*(x,t)
\]
where $\cat G$ is any set of compact generators for $\cat T$.
\end{Thm}

\begin{proof}
Suppose $\frakp \in \min\supp_R\Hom_{\cat T}^*(x,t)$ for some $x\in\cat T^c$. By \cref{lem:suppproperties}(a) we have $\Hom_{\cat T}^*(x,t)_{\frakp}\neq0$ and hence $\supp_R\Hom_{\cat T}^*(x,t)_{\frakp}=\{\frakp\}$ by the minimality of $\frakp$ and \cref{lem:suppproperties}(b). It then follows from \cref{lem:supptorsion} that $\Hom_{\cat T}^*(x,t)_{\frakp}$ is $a$-torsion for any homogeneous element $a\in\frakp$. Now let $r\in\frakp$ be a homogeneous element of degree $d$. Applying $\Hom_{\cat T}^*(-,t)_{\frakp}$ to the exact triangle
\[
x\xra{r}\Sigma^d x\to \kos x{r}
\]
yields a long exact sequence
\begin{multline*}
\cdots\to\Hom_{\cat T}^*(x,t)_{\frakp}[-1]\to\Hom_{\cat T}^*(\kos x{r},t)_{\frakp} \\
\to\Hom_{\cat T}^*(x,t)_{\frakp}[-d]\xra{r}\Hom_{\cat T}^*(x,t)_{\frakp}\to\cdots.
\end{multline*}
Hence $\Hom_{\cat T}^*(\kos x{r},t)_{\frakp}$ is also $a$-torsion for any homogeneous element $a\in\frakp$ and $\Hom_{\cat T}^*(\kos x{r},t)_{\frakp}\neq0$. An induction shows that $\Hom_{\cat T}^*(\kos x{\fraka},t)_{\frakp}\neq0$ for any finitely generated ideal $\fraka\subseteq\frakp$. On the other hand, since $\varGamma_{\cal Z(\frakp)}t\in\cat T_{\cal Z(\frakp)}$ it follows that
\[
\supp_R\Hom_{\cat T}^*(\kos x{\fraka},\varGamma_{\cal Z(\frakp)}t)\subseteq \cal Z(\frakp)
\]
which implies $\supp_R\Hom_{\cat T}^*(\kos x{\fraka},\varGamma_{\cal Z(\frakp)}t)_{\frakp}=\emptyset$ by \cref{lem:suppproperties}(b). In view of \eqref{eq:suppdetection}, we then have $\Hom_{\cat T}^*(\kos x{\fraka},\varGamma_{\cal Z(\frakp)}t)_{\frakp}=0$ and thus
\[
\Hom_{\cat T}^*(\kos x{\fraka},L_{\cal Z(\frakp)}t)_{\frakp}\cong\Hom_{\cat T}^*(\kos x{\fraka},t)_{\frakp}\neq0
\]
by \cref{rem:bik-exact-triangle}. Since $\kos x{\fraka}\in\cat T_{\cal V(\fraka)}$ (\cref{lem:suppkos}), we have
\[
\Hom_{\cat T}^*(\kos x{\fraka},\varGamma_{\cal V(\fraka)}L_{\cal Z(\frakp)}t)_{\frakp}\cong\Hom_{\cat T}^*(\kos x{\fraka},L_{\cal Z(\frakp)}t)_{\frakp}\neq0
\]
and therefore $\varGamma_{\cal V(\fraka)}L_{\cal Z(\frakp)}t\neq0$. This is true for every finitely generated ideal $\fraka\subseteq\frakp$, so $\frakp\in\SuppBIK(t)$, which establishes the first inclusion. For the second, note that
\[
\frakp\notin\bigcup_{x\in \cat G}\Supp_R \Hom_{\cat T}^*(x,t) \iff t\in\cat T_{\cal Z(\frakp)} \iff L_{\cal Z(\frakp)}t=0.
\]
This implies $\frakp\notin\SuppBIK(t)$.
\end{proof}

\begin{Rem}
The theorem above generalizes \cite[Corollary~5.3]{BensonIyengarKrause08}. In \cite[Theorem~5.2]{BensonIyengarKrause08} it was shown that the first inclusion is an equality under the assumption that $R$ is Noetherian. It is not clear if the equality still holds without the Noetherian assumption.
\end{Rem}

\begin{Def}
We say that the BIK support satisfies the \emph{detection property} if $\SuppBIK(t)=\emptyset$ implies $t=0$ for every $t\in\cat T$.
\end{Def}

\begin{Cor}\label{cor:dccbik}
If the descending chain condition holds for the prime ideals of $R$ then the BIK support satisfies the detection property.
\end{Cor}

\begin{proof}
The hypothesis implies that if a subset $S$ of $\Spech(R)$ is nonempty then $\min S$ is nonempty. The statement then follows from \cref{thm:suppbik} and \eqref{eq:suppdetection}.
\end{proof}

\begin{Rem}
If $R$ is Noetherian then the descending chain condition on the prime ideals of $R$ holds (see \cite[Corollary~12.4.5(1)]{DickmannSchwartzTressl19}, for example) and thus the BIK support has the detection property. However, we do not know if the detection property is always satisfied when $R$ is not Noetherian.
\end{Rem}

\begin{Rem}
In the following we record some basic properties of the BIK support, which are inspired by \cite[Theorems~4.2 and 4.7]{BillySanders17pp}.
\end{Rem}

\begin{Prop}\label{prop:suppbik}
The following hold:
\begin{enumerate}
\item $\SuppBIK(0)=\emptyset$.
\item $\SuppBIK(t)=\SuppBIK(\Sigma t)$ for every $t\in\cat T$.
\item $\SuppBIK(c)\subseteq\SuppBIK(a)\cup\SuppBIK(b)$ if $a\to b\to c\to \Sigma a$ is an exact triangle in $\cat T$.
\item $\SuppBIK(t_1\oplus t_2)=\SuppBIK(t_1)\cup\SuppBIK(t_2)$ for any $t_1,t_2\in\cat T$.
\item $\SuppBIK(\varGamma_{\cal V}t)=\cal V\cap\SuppBIK(t)$ and $\SuppBIK(L_{\cal V}t)=\cal V^c\cap\SuppBIK(t)$ for any $t\in\cat T$ and any specialization closed subset $\cal V$ of $\Spech(R)$. 
\end{enumerate}
\end{Prop}

\begin{proof}
Part (a) is clear. Parts (b) and (d) follow from the fact that the localization and colocalization functors preserve finite direct sums and are exact.

For part (c), if $\frakp\notin\SuppBIK(a)\cup\SuppBIK(b)$, then there exist finitely generated ideals $\fraka,\frakb\subseteq\frakp$ such that $\varGamma_{\cal V(\fraka)}L_{\cal Z(\frakp)}a=0$ and $\varGamma_{\cal V(\frakb)}L_{\cal Z(\frakp)}b=0$. Thus by \cref{lem:localcomp} we have
\[
\varGamma_{\cal V(\fraka+\frakb)}L_{\cal Z(\frakp)}a=\varGamma_{\cal V(\fraka)}\varGamma_{\cal V(\frakb)}L_{\cal Z(\frakp)}a=0=\varGamma_{\cal V(\fraka)}\varGamma_{\cal V(\frakb)}L_{\cal Z(\frakp)}b=\varGamma_{\cal V(\fraka+\frakb)}L_{\cal Z(\frakp)}b.
\]
It then follows from the exactness of localization functors that $\varGamma_{\cal V(\fraka+\frakb)}L_{\cal Z(\frakp)}c=0$. Therefore $\frakp\notin\SuppBIK(c)$, which establishes (c).

For (e), note that $\varGamma_{\Spech(R)}L_{\cal V}(\varGamma_{\cal V}t)=0=\varGamma_{\cal V}L_{\emptyset}(L_{\cal V}t)$. Thus $\SuppBIK(\varGamma_{\cal V}t)\subseteq \cal V$ and $\SuppBIK(L_{\cal V}t)\subseteq \cal V^c$. Applying (c) to the exact triangle $\varGamma_{\cal V}t\to t\to L_{\cal V}t$ we obtain
\[\SuppBIK(\varGamma_{\cal V}t)\subseteq\SuppBIK(t)\cup\SuppBIK(L_{\cal V}t)\subseteq\SuppBIK(t)\cup\cal V^c.
\]
Intersecting with $\cal V$ leads to $\SuppBIK(\varGamma_{\cal V}t)=\SuppBIK(t)\cap\cal V$. A similar argument shows that $\SuppBIK(L_{\cal V}t)=\SuppBIK(t)\cap\cal V^c$, which completes the proof.
\end{proof}

\begin{Not}
For any class $\cat E$ of objects in $\cat T$ we write $\Loc{\cat E}$ for the localizing subcategory generated by $\cat E$ and define $\SuppBIK(\cat E) \coloneqq\bigcup_{t\in\cat E}\SuppBIK(t)$.
\end{Not}

\begin{Prop}\label{prop:suppbikloc}
The following hold:
\begin{enumerate}
\item $\SuppBIK(t)$ is localizing closed (recall \cref{rem:loc-top}) for any $t\in\cat T$.
\item $\SuppBIK(\cat L)$ is localizing closed for any localizing subcategory $\cat L$ of $\cat T$.
\item If the BIK support satisfies the detection property then for any class of objects $\cat E$ of $\cat T$ we have
\[
\SuppBIK(\Loc{\cat E})=\overline{\SuppBIK(\cat E)}^{\loc}
\]
where $\overline{\SuppBIK(\cat E)}^{\loc}$ denotes the closure of $\SuppBIK(\cat E)$ in $\Spech(R)$ with respect to the localizing topology.
\item If the BIK support satisfies the detection property then for any set of objects $\{t_i\}_{i\in I}$ of $\cat T$ we have
\begin{equation}\label{eq:suppbik-coproduct}
\SuppBIK(\coprod_{i\in I}t_i)=\overline{\bigcup_{i\in I}\SuppBIK(t_i)}^{\loc}.
\end{equation}
\end{enumerate}
\end{Prop}

\begin{proof}
(a): If $\frakp\not\in\SuppBIK(t)$ then by \cref{prop:suppbikdef} there exist Thomason subsets $\cal V$ and $\cal Z$ with $\frakp\in\cal V\cap\cal Z^c$ and $\varGamma_{\cal V}L_{\cal Z}t=0$. Hence $\cal V\cap \cal Z^c$ is an open neighborhood of $\frakp$ (with respect to the localizing topology) for which $\cal V\cap \cal Z^c \cap \SuppBIK(t)=\emptyset$. Therefore, $\SuppBIK(t)$ is localizing closed.

(b): We write $S$ for $\SuppBIK(\cat L)$. For every $\frakp\in S$ we choose an object $t(\frakp)\in\cat L$ such that $\frakp\in\SuppBIK(t(\frakp))$. Now we have
\[
S=\bigcup_{\frakp\in S}\SuppBIK(t(\frakp))=\SuppBIK(\coprod_{\frakp\in S}t(\frakp)),
\]
where the second equality follows from \cref{prop:suppbik}(d) and that $\coprod_{\frakp\in S}t(\frakp)$ belongs to $\cat L$. Therefore, $\SuppBIK(\cat L)$ is localizing closed by (a).

(c): Note that (b) implies
\[
\SuppBIK(\Loc{\cat E})\supseteq\overline{\SuppBIK(\cat E)}^{\loc}.
\]
If $\frakp\notin\overline{\SuppBIK(\cat E)}^{\loc}$ then there exist Thomason subsets $\cal V$ and $\cal Z$ with $\frakp\in\cal V\cap\cal Z^c$ and $\SuppBIK(\cat E)\cap \cal V\cap\cal Z^c=\emptyset$. By \cref{prop:suppbik}(e) we have
\[
\SuppBIK(\varGamma_{\cal V}L_{\cal Z}(\cat E))=\SuppBIK(\cat E)\cap \cal V\cap\cal Z^c=\emptyset.
\]
The detection property then implies $\varGamma_{\cal V}L_{\cal Z}(\cat E)=0$ and hence $\varGamma_{\cal V}L_{\cal Z}(\Loc{\cat E})=0$, which implies $\frakp\notin\SuppBIK(\Loc{\cat E})$. We thus obtain
\[
\SuppBIK(\Loc{\cat E})=\overline{\SuppBIK(\cat E)}^{\loc}.
\]
(d): Let $\{t_i\}_{i\in I}$ be a set of objects in $\cat T$. Observe that
\begin{align*}
\overline{\bigcup_{i\in I}\SuppBIK(t_i)}^{\loc} &= \SuppBIK(\Loc{t_i\mid i\in I})=\SuppBIK(\Loc{\coprod_{i\in I}t_i}) \\
&=\overline{\SuppBIK(\coprod_{i\in I}t_i)}^{\loc} = \SuppBIK(\coprod_{i\in I}t_i).\qedhere
\end{align*}
\end{proof}

\begin{Exa}\label{eg:biktt}
If $\cat T$ is a rigidly-compactly generated tensor triangulated category then the graded endomorphism ring $\End^*_{\cat T}(\unit)$ of the unit object is graded-commutative (see \cite[Example~2.5(2)]{DellAmbrogioStevenson13}) and it canonically acts on $\cat T$ (see \cite[Section~7]{BensonIyengarKrause11b}). Therefore, we can always use this canonical action to obtain a BIK support theory on $\cat T$. Moreover, the tensor structure gives another way of describing the BIK support as follows.
\end{Exa}

\begin{Not}
Recall that a localizing ideal $\cat L$ of a tensor triangulated category $\cat T$ is a localizing subcategory of $\cat T$ with the property that $\cat L \otimes \cat T \subseteq \cat L$. For any class $\cat E$ of objects in a tensor triangulated category $\cat T$ we write $\Loco{\cat E}$ for the localizing ideal generated by $\cat E$.
\end{Not}

\begin{Prop}\label{prop:suppbiksmash}
Let $\cat T$ be a rigidly-compactly generated tensor triangulated category equipped with the canonical action by $\End^*_{\cat T}(\unit)$. For any object $t\in\cat T$ we have
\[
\SuppBIK(t)=\left\{\frakp\in\Spech(R)\left|
{\begin{gathered}
\varGamma_{\cal V(\fraka)}\unit\otimes L_{\cal Z(\frakp)}\unit\otimes t\neq0\text{ for any finitely} \\
\text{ generated ideal }\fraka\text{ contained in }\frakp
\end{gathered}}\right.\right\}.
\]
\end{Prop}

\begin{proof}
By \cref{prop:bikalgebraic} we have
\[
\cat T_{\cal V(\fraka)}=\Loc{\kos x{\fraka}\mid x\in\cat T^c}
\]
for any finitely generated ideal $\fraka$ and
\[
\cat T_{\cal Z(\frakp)}=\Loc{\kos x{r}\mid x\in\cat T^c \text{ and } r\in R\setminus \frakp\text{ homogeneous}}
\]
for any prime ideal $\frakp$. Note that for any $x\in\cat T^c$ we have $\kos x{\fraka} \simeq x \otimes \kos \unit{\fraka}$. This follows from the definition of Koszul object and the exactness of tensor product. Hence $\Loc{\kos x{\fraka}\mid x\in\cat T^c} = \Loco{\kos \unit{\fraka}}$. Therefore $\cat T_{\cal V(\fraka)}$ is a compactly generated localizing ideal of $\cat T$ and thus a smashing localizing ideal of $\cat T$ by \cite[Theorem~4.1]{BalmerFavi11}. Similarly, $\cat T_{\cal Z(\frakp)}$ is a smashing localizing ideal of $\cat T$. It then follows from \cite[Theorem~2.13]{BalmerFavi11} that the functor $\varGamma_{\cal V(\fraka)}$ is isomorphic to $\varGamma_{\cal V(\fraka)}\unit\otimes-$ and that the functor $L_{\cal Z(\frakp)}$ is isomorphic to $L_{\cal Z(\frakp)}\unit\otimes-$. This finishes the proof.
\end{proof}

\begin{Rem}
An immediate consequence of the proposition above is that
\begin{equation}\label{eq:suppbik-tensor}
\SuppBIK(t_1\otimes t_2)\subseteq \SuppBIK(t_1)\cap \SuppBIK(t_2)
\end{equation}
for any objects $t_1,t_2\in\cat T$. It then follows that the space of supports $\SuppBIK(\cat T)$ coincides with $\SuppBIK(\unit)$. Moreover, \cref{prop:suppbik}(a)(b)(c), \eqref{eq:suppbik-coproduct}, and \eqref{eq:suppbik-tensor} are equivalent to the statement that $\SET{t \in \cat T}{\SuppBIK(t) \subseteq Y}$ is a localizing ideal of~$\cat T$ for any localizing closed subset $Y \subseteq \Spech(\End^*_{\cat T}(\unit))$; \cf \cite[Remark~2.12]{BarthelHeardSander23b}. In \cref{sec:comparison} we will use the canonical BIK support to give a notion of stratification; see \cref{def:coh-strat}.
\end{Rem}

\section{Tensor triangular support}\label{sec:supp}
In this section, after recalling some basic tensor triangular geometry, we summarize some basic properties of the tensor triangular support given in \cite{BillySanders17pp} and then study how this support behaves under geometric functors. We assume some familiarity with \cite{BarthelHeardSander23b} and follow their terminology and notation.

\begin{Hyp}
For the rest of the paper, we fix a rigidly-compactly generated tensor triangulated category $\cat T$.
\end{Hyp}

\begin{Not}
The subcategory $\cat T^c$ of compact objects of $\cat T$ is a rigid, essentially small tensor triangulated category, whose Balmer spectrum is denoted by $\Spc(\cat T^c)$.
\end{Not}

\begin{Rec}
The Balmer support function is defined as
\[
\supp(a) \coloneqq \SET{\cat P \in \Spc(\cat T^c)}{a \notin \cat P}
\]
for any compact object $a \in \cat T^c$. The Thomason closed subsets of $\Spc(\cat T^c)$ are precisely the subsets of the form $\supp(a)$ for $a \in \cat T^c$ and these subsets form a basis of closed subsets for $\Spc(\cat T^c)$. Moreover, $\supp$ induces a bijection
\begin{equation}\label{eq:ttgthm}
\{\text{thick ideals of }\cat T^c\} \xra{\sim} \{\text{Thomason subsets of }\Spc(\cat T^c)\}
\end{equation}
sending a thick ideal $\cat C\subseteq \cat T^c$ to $\supp(\cat C) \coloneqq \bigcup_{a\in \cat C}\supp(a)$. The inverse is given by sending a Thomason subsets $Y$ to $\cat T^c_Y \coloneqq \SET{a \in \cat T^c}{\supp(a) \subseteq Y}$. See \cite{Balmer05a} for more discussion.
\end{Rec}

\begin{Ter}
A coproduct-preserving tensor triangulated functor $F\colon \cat C\to \cat D$ between rigidly-compactly generated tt-categories is called a \emph{geometric functor}. Such a functor enjoys nice properties (\cite{BalmerDellAmbrogioSanders16}). For example, $F$ admits a right adjoint $U$ and we have the projection formula
\begin{equation}\label{eq:projection}
U(F(x)\otimes y) \simeq x \otimes U(y)
\end{equation}
for all $x\in \cat C$ and $y\in \cat D$.
\end{Ter}

\begin{Rec}
Each Thomason subset $Y\subseteq\Spc(\cat T^c)$ gives rise to an idempotent triangle (or equivalently, a smashing localization)
\begin{equation}\label{eq:idempotent-triangle}
e_Y\to\unit\to f_Y\to \Sigma e_Y
\end{equation}
such that
\begin{equation}\label{eq:kerfY}
\Ker(f_Y\otimes-)=\Loco{e_Y}=e_Y\otimes\cat T=\Loc{\cat T^c_Y}.
\end{equation}
Let $V$ be the complement of $Y$ in $\Spc(\cat T^c)$. The associated finite localization
\[
\cat T \to \cat T(V)\coloneqq \cat T/\Loco{e_Y} \cong f_Y \otimes \cat T
\]
is a geometric functor, which gives rise to an embedding $\varphi \colon \Spc(\cat T(V)^c) \hook \Spc(\cat T^c)$ with image $V$; see \cite[Remark~1.23]{BarthelHeardSander23b} for more discussion.
\end{Rec}

\begin{Exa}\label{exa:local-cat}
Let $\cat P\in\Spc(\cat T^c)$. We write $Y_{\cat P}\coloneqq\gen(\cat P)^c$ for the complement of the set of generalizations of $\cat P$.\footnote{In the previous sections, for a prime ideal $\frakp$ in a Zariski spectrum, we used the notation $\cal Z(\frakp)$ for the complement of $\gen(\frakp)$.} Note that $Y_{\cat P}=\supp(\cat P)$ and hence
\begin{equation}\label{eq:LocP}
\Ker(f_{Y_{\cat P}}\otimes-)=\Loco{e_{Y_\cat P}}=\Loc{\cat T^c_{Y_\cat P}}=\Loc{\cat P}
\end{equation}
by \eqref{eq:kerfY} and \eqref{eq:ttgthm}. Therefore, the finite localization $\cat T\to  \cat T(\gen(\cat P)) \eqqcolon \cat T_{\cat P}$ is obtained by killing $\Loc{\cat P}$. The induced spectral map $\varphi_{\cat P}\colon\Spc(\cat T_{\cat P}^c)\hookrightarrow\Spc(\cat T^c)$ identifies $\Spc(\cat T_{\cat P}^c)$ with its image $\gen(\cat P)$. The category $\cat T_{\cat P}$ is called the \emph{localization of $\cat T$ at $\cat P$}. We write $t_{\cat P}$ for the image of an object $t\in\cat T$ in $\cat T_{\cat P}$. See \cite[Remark~1.23 and Definition~1.25]{BarthelHeardSander23b} for further discussion.
\end{Exa}

\begin{Not}
Let $W=U\cap V^c$ be a weakly visible subset of $\Spc(\cat T^c)$ where $U$ and $V$ are Thomason subsets. Define $g_W\coloneqq e_U\otimes f_V$. By \cite[Remark~7.6]{BalmerFavi11} the idempotent object $g_W$ does not depend on the choice of $U$ and $V$ up to isomorphism. Moreover, it follows from \cite[Lemma~1.27]{BarthelHeardSander23b} that
\begin{equation}\label{eq:gW0}
g_W=0 \iff W=\emptyset.
\end{equation}
\end{Not}

\begin{Rem}
Let $F\colon\cat C\to\cat D$ be a geometric functor between rigidly-compactly generated tensor triangulated categories and let $\varphi\colon\Spc(\cat D^c)\to\Spc(\cat C^c)$ be the induced spectral map. By \cite[Proposition~5.11]{BalmerSanders17} we have
\begin{equation}\label{eq:FgW}
F(g_W)\simeq g_{\varphi\inv(W)}
\end{equation}
for any weakly visible subset $W$ of $\Spc(\cat C^c)$.
\end{Rem}

\begin{Rem}
A point $\cat P$ in the Balmer spectrum $\Spc(\cat T^c)$ is not always weakly visible. However, since the Thomason closed subsets of $\Spc(\cat T^c)$ form a basis of closed subsets it follows that 
\[
\overline{\{\cat P\}}=\bigcap_{\substack{a\in \cat T^c \\ a\notin \cat P}}\supp(a) \quad \text{and hence} \quad \{\cat P\} = \bigcap_{\substack{a\in \cat T^c \\ a\notin \cat P}}\supp(a) \cap \gen(\cat P).
\]
In other words, $\{\cat P\}$ is an intersection of weakly visible subsets. This leads to the following:
\end{Rem}

\begin{Def}[W. Sanders]\label{def:Supp}
The \emph{tensor triangular support} of an object $t\in \cat T$ is
\[
\Supp(t)\coloneqq \left\{ \cat P\in\Spc(\cat T^c) \left|
\begin{gathered}
g_W\otimes t\neq0\text{ for every weakly}\\
\text{ visible subset }W\text{ containing }\cat P
\end{gathered} \right. \right\}.
\]
\end{Def}

\begin{Rem}\label{rem:Supp-def}
By \cite[Lemma~1.27]{BarthelHeardSander23b} we have
\begin{equation}\label{eq:gW12}
g_{W_1\cap W_2}=g_{W_1}\otimes g_{W_2}
\end{equation}
for any weakly visible subsets $W_1$ and $W_2$. By \cref{rem:Thomason-closed} every Thomason subset of $\Spc(\cat T^c)$ is a union of Thomason closed subsets and $Y_{\cat P} = \gen(\cat P)^c$ is the largest Thomason subset not containing $\cat P$. Therefore, a Balmer prime $\cat P$ is in $\Supp(t)$ if and only if $e_{\supp a}\otimes f_{Y_{\cat P}}\otimes t\neq0$ for every Thomason closed subset $\supp a$ that contains $\cat P$ (\cf \cref{def:SuppBIK}).
\end{Rem}

\begin{Rem}
If $\Spc(\cat T^c)$ is weakly Noetherian, that is, every point of $\Spc(\cat T^c)$ is weakly visible, then the tensor triangular support coincides with the Balmer--Favi support in \cite[Definition~2.11]{BarthelHeardSander23b}. Indeed, if $\cat P\in\Spc(\cat T^c)$ is weakly visible then \eqref{eq:gW12} implies that $\cat P\in\Supp(t)$ if and only if $g_{\cat P}\otimes t\neq0$.
\end{Rem}

\begin{Prop}\label{prop:Supp-prop}
The tensor triangular support has the following basic properties:
\begin{enumerate}
\item $\Supp(0) = \emptyset$ and $\Supp(\unit) = \Spc(\cat T^c)$.
\item $\Supp(\Sigma t) = \Supp(t)$ for every $t \in\cat T$.
\item $\Supp(c) \subseteq \Supp(a) \cup \Supp(b)$ for any exact triangle $a \to b \to c \to \Sigma a$ in~$\cat T$.
\item $\Supp(t_1\oplus t_2)=\Supp(t_1)\cup \Supp(t_2)$ for any $t_1,t_2\in\cat T$.
\item $\Supp(t_1 \otimes t_2) \subseteq \Supp(t_1) \cap \Supp(t_2)$ for any $t_1,t_2 \in \cat T$.
\item $\Supp(t\otimes e_Y)=\Supp(t)\cap Y$ and $\Supp(t\otimes f_Y)=\Supp(t)\cap Y^c$ for every object $t\in\cat T$ and every Thomason subset $Y\subseteq \Spc(\cat T^c)$.
\end{enumerate}
\end{Prop}

\begin{proof}
Parts (a), (b), (d), and (e) are immediate from the definitions. The proofs for parts (c) and (f) can be found in \cite[Theorem~4.2]{BillySanders17pp}.
\end{proof}

\begin{Rem}
The half-tensor formula \cite[Lemma~2.18]{BarthelHeardSander23b} still holds without the weakly Noetherian assumption:
\end{Rem}

\begin{Lem}\label{lem:half-tensor}
For any compact object $x\in \cat T^c$ and arbitrary object $t \in \cat T$ we have
\[ 
\Supp(x \otimes t) = \supp(x) \cap \Supp(t).
\]
In particular, for any compact object $x \in \cat T^c$, the tensor triangular support coincides with the usual notion of support: $\Supp(x) = \supp(x)$.
\end{Lem}

\begin{proof}
Observe that
\begin{align*}
\cat P\notin\Supp(x\otimes t) & \iff x\otimes t\otimes g_W=0 \text{ for some weakly visible }W\ni \cat P \\
& \iff e_{\supp(x)}\otimes t\otimes g_W=0 \quad (\Loco{x}=e_{\supp(x)}\otimes \cat T) \\
& \iff \cat P\notin\Supp(e_{\supp(x)}\otimes t)=\supp(x)\cap\Supp(t)
\end{align*}
where the last equality is due to \cref{prop:Supp-prop}(f).
\end{proof}

\begin{Rem}
Recall from \cref{rem:loc-top} that $\Spc(\cat T^c)$ is weakly Noetherian if and only if its localizing topology is discrete. Thus the localizing topology becomes relevant if $\Spc(\cat T^c)$ is not weakly Noetherian, as the following shows.
\end{Rem}

\begin{Prop}\label{prop:Supp-loc-closed}
For any object $t\in\cat T$ and any localizing subcategory $\cat L$ of $\cat T$ we have
\begin{enumerate}
\item $\Supp(t)$ is localizing closed.
\item $\Supp(\cat L) \coloneqq \bigcup_{s \in \cat L}\Supp(s)$ is localizing closed.
\end{enumerate}
\end{Prop}

\begin{proof}
See \cite[Theorems~4.2]{BillySanders17pp}.
\end{proof}

\begin{Rem}
We now study how the tensor triangular support behaves under geometric functors.
\end{Rem}

\begin{Prop}\label{prop:Supp-U1}
Let $F \colon \cat C \to \cat D$ be a geometric functor between rigidly-compactly generated tt-categories, $U$ its right adjoint, and $\varphi \colon \Spc(\cat D^c)\to\Spc(\cat C^c)$ the induced spectral map. We have
\[
\Supp(U(\unit))=\overline{\im\varphi}^{\loc}
\]
where $\overline{\im\varphi}^{\loc}$ denotes the closure of $\im\varphi$ in $\Spc(\cat C^c)$ with respect to the localizing topology.
\end{Prop}

\begin{proof}
If $\cat Q\notin\Supp(U(\unit))$ then there exists some weakly visible subset $W\ni\cat Q$ such that $g_W\otimes U(\unit)=0$ and thus $U(F(g_W))=0$ by the projection formula \eqref{eq:projection}. It follows that $0=\Hom(g_W,UF(g_W))\simeq\Hom(F(g_W),F(g_W))$ and hence $F(g_W)=0$. We then have $\varphi\inv(W)=\emptyset$ in view of \eqref{eq:gW0} and \eqref{eq:FgW}. Therefore $\cat Q\notin\im\varphi$. We have established $\im\varphi \subseteq \Supp(U(\unit))$ and it follows from \cref{prop:Supp-loc-closed} that $\overline{\im\varphi}^{\loc} \subseteq \Supp(U(\unit))$. For the other inclusion, if $\cat Q\notin\overline{\im\varphi}^{\loc}$ then there exists a weakly visible subset $W\ni\cat Q$ such that $W\cap\im\varphi=\emptyset$. This means $\varphi\inv(W)=\emptyset$, which implies $F(g_W)=0$. By the projection formula we obtain $g_W\otimes U(\unit)=0$ and hence $\cat Q\notin\Supp(U(\unit))$.
\end{proof}

\begin{Rem}
In \cite[Corollary~13.15]{BCHS25} it was shown that if $\Spc(\cat C^c)$ is weakly Noetherian then $\Supp(U(\unit))=\im\varphi$, which is a special case of the proposition above.
\end{Rem}

\begin{Prop}\label{prop:Supp-basechange}
Let $F$, $U$, and $\varphi$ be as in \cref{prop:Supp-U1}. The following hold:
\begin{enumerate}
	\item $\Supp(F(c)) \subseteq \varphi\inv(\Supp(c))$ for any $c\in\cat C$.
	\item $\varphi(\Supp(d))\subseteq\Supp(U(d))$ for any $d\in\cat D$ if $U$ is conservative.
\end{enumerate}
\end{Prop}

\begin{proof}
For part (a), let $\cat P\in\Supp(F(c))$. If $\cat Q\coloneqq \varphi(\cat P)\notin\Supp(c)$ then there exists a weakly visible subset $W\ni\cat Q$ with $g_W\otimes c=0$. Hence
\[
0=F(g_W\otimes c)\simeq F(g_W)\otimes F(c)\simeq g_{\varphi\inv(W)}\otimes F(c)
\]
which contradicts $\cat P\in\Supp(F(c))$. This establishes (a).

To prove part (b), suppose that $\cat P\in\Supp(d)$ but $\cat Q\coloneqq\varphi(\cat P)\notin\Supp(U(d))$. By definition there exists a weakly visible subset $W\ni\cat Q$ such that $g_W\otimes U(d)=0$. We then have $U(F(g_W)\otimes d)=0$ by the projection formula and thus $g_{\varphi\inv(W)}\otimes d=0$ by the conservativity of $U$. \cref{prop:Supp-prop}(f) then implies $\varphi\inv(W)\cap\Supp(d)=\emptyset$, which contradicts $\cat P\in\Supp(d)$.
\end{proof}

\begin{Cor}\label{cor:Supp-finite-loc}
Let $F$, $U$, and $\varphi$ be as in \cref{prop:Supp-U1}. If $F$ is a finite localization then we have
\begin{enumerate}
	\item $\varphi(\Supp(d)) = \Supp(U(d))$ for any $d\in\cat D$.
	\item $\Supp(F(c)) = \varphi\inv(\Supp(c))$ for any $c\in\cat C$.
\end{enumerate}
\end{Cor}

\begin{proof}
For part (a), it suffices to show that $\Supp(U(d)) \subseteq \varphi(\Supp(d))$ by \cref{prop:Supp-basechange}(b). To this end, suppose $\cat Q = \varphi(\cat P) \in \Supp(U(d))$ but $\cat P \notin \Supp(d)$. By definition there exists a weakly visible subset $W \subseteq \Spc(\cat D^c)$ such that $\cat P \in W$ and $g_W\otimes d=0$. The map $\varphi \colon \Spc(\cat D^c)\hook\Spc(\cat C^c)$ exhibits $\Spc(\cat D^c)$ as a spectral subspace of $\Spc(\cat C^c)$ since $F$ is a finite localization. We thus have $W=\varphi\inv(Z)$ for some weakly visible subset $Z\subseteq\Spc(\cat C^c)$ by \cite[Theorem~2.1.3]{DickmannSchwartzTressl19}. It follows from the projection formula and \eqref{eq:FgW} that $g_{Z} \otimes U(d) \simeq U(F(g_{Z})\otimes d) \simeq U(g_W\otimes d)= 0$, which contradicts $\cat Q \in \Supp(U(d))$. This establishes (a).

For part (b), suppose that $F$ is the finite localization associated to a Thomason subset $Y \subseteq \Spc(\cat C^c)$ and let $V$ be the complement of $Y$. We then have an induced embedding $\varphi\colon\Spc(\cat T(V)^c) \hook \Spc(\cat T^c)$ with image $V$. For any $c \in \cat C$ we have $\Supp(F(c))=\varphi\inv(\Supp(UF(c)))$ by part (a). Observe that
\[
\begin{split}
\varphi\inv(\Supp(UF(c)))&=\varphi\inv(\Supp(U(\unit)\otimes c))=\varphi\inv(\Supp(f_Y\otimes c)) \\
&=\varphi\inv(Y^c\cap\Supp(c))=\varphi\inv(\Supp(c))
\end{split}
\]
where the first equality uses the projection formula, the second equality uses the fact that $U(\unit)=UF(\unit)=f_Y$, and the third equality is due to \cref{lem:half-tensor}. This completes the proof.
\end{proof}

\section{The detection property}\label{sec:detection}
Our next goal is to show that the detection property can be checked at the algebraic localizations at the closed points of the Zariski spectrum of the graded endomorphism ring of the unit.

\begin{Exa}[Graded algebraic localization]\label{exa:graded-algebraic-local}
Recall from \cref{eg:biktt} that the graded endomorphism ring $\End_{\cat T}^*(\unit)$ of the unit object canonically acts on $\cat T$. Thus the machinery in \cref{sec:bik} applies. Denote by $\End_{\cat T}^{\mathrm{hom}}(\unit)$ the set of homogeneous elements in $\End_{\cat T}^*(\unit)$. There is a natural continuous map
\[
\begin{split}
\rho\colon\Spc(\cat T^c)&\to\Spech(\End^*_{\cat T}(\unit))\\
\cat P&\mapsto\ideal{f\in\End^{\mathrm{hom}}_{\cat T}(\unit) \mid \cone(f)\notin\cat P}
\end{split}
\]
such that $\rho\inv(\cal V(s))=\supp(\cone(s))$ for any $s\in \End^{\mathrm{hom}}_{\cat T}(\unit)$ \cite[Theorem~5.3]{Balmer10a}. Let $S\subseteq\End^{\mathrm{hom}}_{\cat T}(\unit)$ be a multiplicative subset of homogeneous elements. We denote the finite localization of $\cat T$ associated to the Thomason subset
\[
\rho\inv(\bigcup_{s\in S}\cal V(s))=\bigcup_{s\in S}\supp(\cone(s))
\]
by $S\inv\cat T$, which is called the \emph{algebraic localization of $\cat T$ with respect to $S$}. The corresponding localization functor and colocalization functor are denoted by $L_S$ and $\varGamma_S$, respectively.

In particular, for a prime ideal $\frakp\in\Spech(\End^*_{\cat T}(\unit))$ we define the multiplicative subset $S_{\frakp}\coloneqq\SET{s\in\End_{\cat T}^{\mathrm{hom}}(\unit)}{s\notin\frakp}$ and call $\cat T_{\frakp}\coloneqq S_{\frakp}\inv\cat T$ the \emph{algebraic localization of $\cat T$ at $\frakp$}. Observe that
\[
\Loco{\cone(s)\mid s\in S_{\frakp}}=\Loc{\kos x{s}\mid x\in \cat T^c, s\in S_{\frakp}}=\cat T_{\cal Z(\frakp)}
\]
where the first equality is explained in the proof of \cref{prop:suppbiksmash} and the second equality follows from \cref{prop:bikalgebraic}(b). Therefore, the localization functor $L_{\cal Z(\frakp)}$ in \cref{def:SuppBIK} is identical to the algebraic localization functor $L_{S_{\frakp}}$ which is associated to the Thomason subset $\rho\inv(\bigcup_{s\in S_{\frakp}}\cal V(s))=\rho\inv(\cal Z(\frakp))$. 
\end{Exa}

\begin{Not}
For any $t\in\cat T$ we write $\pi_*(t)$ for
\[
\Hom_{\cat T}^{-*}(\unit,t) = \coprod_{n\in\bbZ}\Hom_{\cat T}(\unit,\Sigma^{-n} t).
\]
\end{Not}

\begin{Rem}
The following proposition generalizes both \cite[Theorem~3.3.7]{HoveyPalmieriStrickland97} and \cite[Corollary~3.10]{Balmer10a}. The ungraded version of it can be found in \cite[Theorem~2.33(h)]{DellAmbrogio10}.
\end{Rem}

\begin{Prop}\label{prop:algebraic-local}
Let $S\subseteq\End^{\mathrm{hom}}_{\cat T}(\unit)$ be a multiplicative subset of homogeneous elements. There is a natural isomorphism
\[
\Hom_{\cat T}^*(x,L_S t)\cong S\inv\Hom_{\cat T}^*(x,t)
\]
for $x\in\cat T^c$ and $t\in\cat T$. In particular, we have
\[
\pi_*(L_S t)\cong S\inv\pi_*(t)
\]
for any $t\in\cat T$.
\end{Prop}

\begin{proof}
If $x$ is compact then $S\inv\Hom^*_{\cat T}(x,-)$ is a coproduct-preserving homological functor on $\cat T$ which vanishes on $t\otimes \cone(s)$ for all $t\in\cat T$ and $s\in S$ and hence on $\Loc{\cat T\otimes \cone(s)\mid s\in S}=\Loco{\cone(s)\mid s\in S}=\varGamma_S \cat T$. It follows from the localization triangle
\[
\varGamma_S t \to t \to L_S t
\]
that there is a natural isomorphism
\[
S\inv\Hom^*_{\cat T}(x,t)\cong S\inv\Hom^*_{\cat T}(x,L_S t)
\]
for any $x\in \cat T^c$ and $t\in \cat T$. Let $d$ be the degree of $s$. Applying $\Hom^*_{\cat T}(-,t)$ to the exact triangle
\[
x\xra{s}\Sigma^d x\to \cone(s)\otimes x
\]
yields an exact sequence
\begin{multline*}
0=\Hom^*_{\cat T}(\cone(s)\otimes x,L_S t)\to\Hom^*_{\cat T}(x,L_S t)[-d] \\
\xra{s}\Hom^*_{\cat T}(x,L_S t)\to\Hom^*_{\cat T}(\cone(s)\otimes x,L_S t)[1]=0.
\end{multline*}
Thus multiplying by $s$ is an isomorphism on $\Hom^*_{\cat T}(x,L_S t)$. Therefore we have
\[
S\inv\Hom^*_{\cat T}(x,L_S t)\cong\Hom^*_{\cat T}(x,L_S t)
\]
which completes the proof.
\end{proof}

\begin{Exa}\label{exa:Spc-SH}
Recall from \cite[Corollary~9.5]{Balmer10a} that the Balmer spectrum $\Spc(\SH^c)$ of the stable homotopy category together with its comparison map can be depicted as follows:
\begin{equation*}
\vcenter{\xymatrix@C=.8em @R=.4em{
&&\cat P_{2,\infty} \ar@{-}[d]
&\cat P_{3,\infty} \ar@{-}[d]
&& \kern-2em{\cdots}
&\cat P_{p,\infty} \ar@{-}[d]
& {\cdots}
\\
\Spc(\SHc)= \ar[ddddddd]_-{\displaystyle\rho_\SHc}
&&{\vdots} \ar@{-}[d]
& {\vdots} \ar@{-}[d]
&&& {\vdots} \ar@{-}[d]
\\
&&\cat P_{2,n+1} \ar@{-}[d]
& \cat P_{3,n+1} \ar@{-}[d]
&& \kern-2em{\cdots}
& \cat P_{p,n+1} \ar@{-}[d]
& {\cdots}
\\
&&\cat P_{2,n} \ar@{-}[d]
& \cat P_{3,n} \ar@{-}[d]
&& \kern-2em{\cdots}
& \cat P_{p,n} \ar@{-}[d]
& {\cdots}
\\
&&{\vdots} \ar@{-}[d]
& {\vdots} \ar@{-}[d]
&&& {\vdots} \ar@{-}[d]
\\
&&\cat P_{2,1} \ar@{-}[rrd]
& \cat P_{3,1} \ar@{-}[rd]
&& \kern-2em{\cdots}
& \cat P_{p,1} \ar@{-}[lld]
& {\cdots}
\\
&&&& \cat P_{p,0}
\\\\
\Spec(\bbZ)=
&&2\bbZ \ar@{-}[rrd]
& 3\bbZ \ar@{-}[rd]
&& \kern-2em{\cdots}
& p\bbZ \ar@{-}[lld]
& {\cdots}
\\
&&&& (0)}}
\end{equation*}
where $\cat P_{p,n}$ is the kernel in $\SHc$ of the $p$-local $n$-th Morava K-theory. In other words, $\cat P_{p,n} = \SET{x \in \SHc}{x\otimes K(p,n) = 0}$ where $K(p,n)$ denotes the $p$-local $n$-th Morava K-theory spectrum. In particular,~$\cat P_{p,0}=\SH_{\mathrm{tor}}^c$ is the subcategory of finite torsion spectra, which is independent of $p$. For any prime number $p$, the $p$-local stable homotopy category $\SHp$, which is defined as the Bousfield localization of the stable homotopy category $\SH$ with respect to the mod-$p$ Moore spectrum, can be realized as the algebraic localization at $p\bbZ\in\Spech(\End_{\SH}^*(\unit))\cong\Spec(\bbZ)$ of $\SH$ by \cref{prop:algebraic-local}. Moreover, since we have $\cat P_{p,\infty}=\thickt{\cone(s)\mid s\notin p\bbZ}$ by \cite[Corollary~9.5(c)]{Balmer10a}, the algebraic localization at $p\bbZ$ coincides with the localization at $\cat P_{p,\infty}$ in the sense of \cref{exa:local-cat}. Therefore, the Balmer spectrum $\Spc(\SHcp)$ can be identified with
\[
\gen(\cat P_{p,\infty})=\{\cat P_{p,n}\mid0\le n\le\infty\}\subset\Spc(\SHc).
\]
For any spectrum $t\in\SH$, we write $t_{(p)}\coloneqq t_{p\bbZ} = t_{\cat P_{p,\infty}}$ for the corresponding $p$-local spectrum.
\end{Exa}

\begin{Exa}\label{exa:SH-non-weakly-noetherian}
The space $\Spc(\SHc)$ is not weakly Noetherian. Indeed, by \cite[Corollary~9.5]{Balmer10a} any Thomason subset of $\Spc(\SHc)$ is the union of subsets of the form $\overline{\{\cat P_{p,n_p}\}}$ where $p$ is a prime number and $0\le n_p<\infty$. Thus the closed point $\cat P_{p,\infty}$ is not weakly visible for every prime number $p$. In particular, any Thomason subset of $\Spc(\SHcp)$ is of the form $\overline{\{\cat P_{p,n}\}}$ for $0\le n<\infty$ and therefore $\cat P_{p,\infty}$ is the only point in $\Spc(\SHcp)$ that is not weakly visible.
\end{Exa}

\begin{Rem}\label{rem:SH-detection}
In \cite[Remark~11.11]{BarthelHeardSander23b} the authors considered the Balmer--Favi support for $\SHp$ that excludes $\cat P_{p,\infty}$ since it is not weakly visible. More precisely, for any $t_{(p)} \in \SHp$ they defined
\[
\Supp_{<\infty}(t_{(p)})\coloneqq \SET{\cat P_{p,n}}{0\le n<\infty \text{ and } g_{\cat P_{p,n}}\otimes t_{(p)} \neq 0}.
\]
They further extended this support to the point $\cat P_{p,\infty}$ by declaring that $\cat P_{p,\infty}$ is in the support of a $p$-local spectrum $t_{(p)}$ if and only if $\HFp\otimes t_{(p)}\neq0$. We denote this extended support function by $\Supp_{\le\infty}$. Let $I\in\SH$ be the Brown-Comenetz dual of the sphere spectrum. Note that the $p$-local Brown-Comenetz dual of the sphere spectrum $I_{(p)}\in\SHp$ is isomorphic to the Brown-Comenetz dual of the $p$-local sphere spectrum as defined in \cite[Section~7]{HoveyPalmieri99}. In \cite[Remark~11.11]{BarthelHeardSander23b} it was explained that $\Supp_{<\infty}(I_{(p)})=\emptyset$ and $\HFp\otimes I_{(p)}=0$. Hence $\Supp_{\le\infty}(I_{(p)})=\emptyset$. This motivates the following:
\end{Rem}

\begin{Def}[The detection property]
\label{def:detectionproperty}
We say that $\cat T$ has the \emph{detection property} if $\Supp(t)=\emptyset$ implies $t=0$ for every $t \in \cat T$.
\end{Def}

\begin{Rem}
If the space $\Spc(\cat T^c)$ is Noetherian then $\cat T$ has the detection property by \cite[Theorem~3.22 and Remark~3.9]{BarthelHeardSander23b}. For example, the derived category $\Der(R)$ of a commutative Noetherian ring has the detection property. However, we do not know in what generality the detection property holds; see, for example, the discussion in \cite[Section~8.1]{BillySanders17pp} and \cite[Remark~6.6]{BCHS25}.
\end{Rem}

\begin{Exa}
In \cref{rem:SH-detection} we see that the function $\Supp_{\le\infty}$ does not detect vanishing of objects in $\SHp$. However, the tensor triangular support $\Supp$ does:
\end{Exa}

\begin{Prop}\label{prop:SHp-detection}
The category $\SHp$ satisfies the detection property.
\end{Prop}

\begin{proof}
Suppose $\Supp(t_{(p)})=\emptyset$ for some $t_{(p)} \in \SHp$. Since $\cat P_{p,\infty} \notin \Supp(t_{(p)})$, we have $e_{\overline{\{\cat P_{p,n}\}}} \otimes t_{(p)} = 0$ for some $0\le n < \infty$ and thus $t_{(p)} \simeq f_{\overline{\{\cat P_{p,n}\}}} \otimes t_{(p)}$. If $n=0$ then $t_{(p)} = 0$ as desired. Now suppose $n>0$. By $\cat P_{p,n-1} \notin \Supp(t_{(p)})$ we have
\[
0 = e_{\overline{\{\cat P_{p,n-1}\}}} \otimes f_{\overline{\{\cat P_{p,n}\}}} \otimes t_{(p)} \simeq e_{\overline{\{\cat P_{p,n-1}\}}} \otimes t_{(p)}.
\]
An induction shows that $0 = e_{\overline{\{\cat P_{p,0}\}}} \otimes t_{(p)} \simeq t_{(p)}$, which completes the proof.
\end{proof}

\begin{Exa}\label{exa:Supp-Ip}
By \cref{rem:SH-detection} and \cref{prop:SHp-detection}, the $p$-local Brown-Comenetz dual of the sphere spectrum has support $\Supp(I_{(p)})=\{\cat P_{p,\infty}\}$.
\end{Exa}

\begin{Rem}
The following corollary says that the tensor triangular support of an object can be computed at its localizations at all closed points in the Balmer spectrum. Keep in mind the notation introduced in \cref{exa:local-cat}.
\end{Rem}

\begin{Cor}\label{cor:supplocal}
For every object $t\in\cat T$ we have
\[
\Supp(t)=\bigcup_{
\substack{
\cat M\in\Spc(\cat T^c) \\
\cat M\text{ closed}
}}\varphi_{\cat M}(\Supp(t_{\cat M}))
\]
where $\varphi_{\cat M}\coloneqq \Spc(\cat T_{\cat M}^c)\hookrightarrow\Spc(\cat T^c)$ is the map induced by the localization at $\cat M$.
\end{Cor}

\begin{proof}
By \cref{cor:Supp-finite-loc}(b), for any $t\in\cat T$ we have $\Supp(t_{\cat M})=\varphi_{\cat M}\inv(\Supp(t))$ and therefore $\varphi_{\cat M}(\Supp(t_{\cat M}))=\Supp(t)\cap\im\varphi_{\cat M}=\Supp(t)\cap \gen(\cat M)$. The result then follows from \cite[Proposition~2.11]{Balmer05a}.
\end{proof}

\begin{Exa}\label{exa:Supp-I}
From \cref{exa:Spc-SH}, \cref{exa:Supp-Ip}, and \cref{cor:supplocal} we obtain that $\Supp(I)=\SET{\cat P_{p,\infty}}{p\text{ prime}}$. In other words, the Brown-Comenetz dual of the sphere spectrum $I$ is supported on the line at infinity in the picture of $\Spc(\SHc)$.
\end{Exa}

\begin{Thm}[Detection property is algebraically local]\label{thm:detectionlocal}
The category $\cat T$ has the detection property if and only if for every closed point $\mathfrak m\in\Spech(\End^*_{\cat T}(\unit))$ the algebraic localization $\cat T_{\mathfrak m}$ has the detection property .
\end{Thm}

\begin{proof}
If $\cat T$ satisfies the detection property then $\cat T_{\frakp}$ has the detection property for every $\frakp\in\Spech(\End^*_{\cat T}(\unit))$ by \cref{cor:Supp-finite-loc}(a). Conversely, suppose that $\cat T_{\mathfrak m}$ satisfies the detection property for every closed point $\mathfrak m\in\Spech(\End^*_{\cat T}(\unit))$. Let $t$ be an object in $\cat T$ with $\Supp(t)=\emptyset$. By \cref{prop:Supp-basechange}(a) we have $\Supp(t_{\mathfrak m})=\emptyset$ and thus $t_{\mathfrak m}=0$. It follows from \cref{prop:algebraic-local} that $\Hom(x,t)_{\mathfrak m}\cong\Hom(x,t_{\mathfrak m})=0$ for every $x\in\cat T^c$. Since $\mathfrak m$ ranges over all the closed points, $t=0$, which completes the proof.
\end{proof}

\begin{Exa}\label{exa:SH-detection}
The stable homotopy category $\SHp$ of $p$-local spectra satisfies the detection property (\cref{prop:SHp-detection}). The theorem above thus implies that the stable homotopy category $\SH$ of all spectra satisfies the detection property.
\end{Exa}

\begin{Cor}\label{cor:dissonant}
Any spectrum whose support is contained in $\SET{\cat P_{p,\infty}}{p \text{ prime}}$ is dissonant.
\end{Cor}

\begin{proof}
According to \cite[Definition~4.1]{Ravenel84}, a spectrum $t\in\SH$ is called dissonant if $t \otimes K(p,n)=0$ for every prime number $p$ and every $n$ with $0 \le n < \infty$. Suppose that $t\in \SH$ is a spectrum with $\Supp(t) \subseteq \SET{\cat P_{p,\infty}}{p \text{ prime}}$. It follows that for any prime number $p$ and any $n$ with $0\le n<\infty$ we have
\[
\emptyset=\Supp(t) \cap \gen(\cat P_{p,n})=\Supp(t \otimes f_{Y_{\cat P_{p,n}}})
\]
where the second equality is due to \cref{prop:Supp-prop}(f). In light of \cref{exa:SH-detection}, we have $t \otimes f_{Y_{\cat P_{p,n}}}=0$, which is equivalent to $t \in \Loc{\cat P_{p,n}}$ by \eqref{eq:LocP}. It then follows from the definition of $\cat P_{p,n}$ that $t \otimes K(p,n) = 0$. This is true for every prime number $p$ and every $n$ with $0 \le n < \infty$. Therefore $t$ is dissonant.
\end{proof}

\section{The local-to-global principle}\label{sec:ltg}
We now introduce a local-to-global principle for a rigidly-compactly generated tensor triangulated category which does not require any topological hypothesis on the Balmer spectrum.

\begin{Def}[The local-to-global principle]\label{def:ltg}
We say that $\cat T$ satisfies the \emph{local-to-global principle} if
\[ 
\Loco{t} = \Loco{t \otimes g_{W_i} \mid i\in I}
\]
for every object $t \in \cat T$ and every cover of $\Spc(\cat T^c)$ by weakly visible subsets $W_i$.
\end{Def}

\begin{Rem}\label{rem:ltg-equivalence}
By \cite[Lemma~3.6]{BarthelHeardSander23b} the definition above is equivalent to
\[
\unit \in \Loco{g_{W_i} \mid i\in I}
\]
for every collection $\{W_i\}_{i\in I}$ of weakly visible subsets such that $\bigcup_{i\in I}W_i=\Spc(\cat T^c)$.
If $\Spc(\cat T^c)$ is weakly Noetherian then every point $\cat P\in\Spc(\cat T^c)$ is weakly visible and thus we can always consider the cover $\bigcup_{\cat P\in\Spc(\cat T^c)}\{\cat P\}=\Spc(\cat T^c)$. In this case, the local-to-global principle is equivalent to
\[
\Loco{t} = \Loco{t \otimes g_{\cat P} \mid \cat P\in \Spc(\cat T^c)}.
\]
Hence it recovers \cite[Definition~3.8]{BarthelHeardSander23b}.
\end{Rem}

\begin{Prop}\label{prop:SHp-ltg}
The $p$-local stable homotopy category $\cat S \coloneqq \SHp$ satisifies the local-to-global principle.
\end{Prop}

\begin{proof}
For any weakly visible cover $\bigcup_{i\in I}W_i=\Spc(\cat S^c)$, there exists some $j\in I$ such that $\cat P_{p,\infty}\in W_j$. If $W_j=\Spc(\cat S^c)$ then $\unit=g_{W_j} \in \Loco{g_{W_i} \mid i\in I}$. Now we assume $W_j$ is proper in $\Spc(\cat S^c)$. We then have $W_j = \overline{\{\cat P_{p,n+1}\}}=\supp(\cat P_{p,n})$ for some $0 \le n<\infty$ and therefore $g_{W_j} \simeq e_{\overline{\{\cat P_{p,n+1}\}}}$. Let $F_n \colon \cat S \to \cat S_n$ denote the finite localization associated to the Thomason subset $W_j$. We write $e_n$ for $e_{\overline{\{\cat P_{p,n+1}\}}}$ and~$f_n$ for $f_{\overline{\{\cat P_{p,n+1}\}}}$. Since $\Spc(\cat S^c_n)$ consists of only finitely many points, $\cat S_n$ satisfies the local-to-global principle by \cite[Theorem~3.22]{BarthelHeardSander23b}. Therefore
\[
F_n(\unit) \in \Loco{F_n(f_n\otimes g_{W_i}) \mid i \in I}.
\]
By \cite[Lemma~3.16]{BarthelHeardSander23b} we obtain
\[
\unit \in \Loco{e_n,\{f_n\otimes g_{W_i}\}_{i \in I}} = \Loco{g_{W_i} \mid i \in I}.\qedhere
\]
\end{proof}

\begin{Rem}\label{rem:ltg-detection}
The local-to-global principle implies the detection property: Let $t\in \cat T$ be an object with $\Supp(t)=\emptyset$. For every point $\cat P\in\Spc(\cat T^c)$ there exists a weakly visible subset $W_{\cat P}\ni\cat P$ with $t \otimes g_{W_{\cat P}}=0$. By the local-to-global principle we have
\[
t \in \Loco{ t \otimes g_{W_{\cat P}} \mid \cat P \in \Spc(\cat T^c) }=0
\]
which forces $t=0$.
\end{Rem}

\begin{Rem}\label{rem:Hochweakscatter}
Recall from \cref{def:hoch-scatter} that the Balmer spectrum $\Spc(\cat T^c)$ is said to be Hochster weakly scattered if its Hochster dual is weakly scattered. This means that for every proper Thomason subset $Y\subsetneq \Spc(\cat T^c)$, there exist a point $\cat P\notin Y$ and a Thomason subset $U\subseteq \Spc(\cat T^c)$ such that
\[
\cat P\in U\cap Y^c\subseteq \gen(\cat P).
\]
\end{Rem}

\begin{Thm}\label{thm:Hochweakscatter}
If $\Spc(\cat T^c)$ is Hochster weakly scattered then $\cat T$ satisfies the local-to-global principle.
\end{Thm}

\begin{proof}
Let $\{W_i\}_{i\in I}$ be a cover of $\Spc(\cat T^c)$ by weakly visible subsets. Consider the localizing ideal $\cat L \coloneqq \Loco{g_{W_i}\mid i\in I}$ and define $Y\coloneqq\bigcup_{a\in\cat L\cap \cat T^c}\supp(a)$. Note that $Y$ is a Thomason subset and $e_Y\in\Loco{e_{\supp(a)}\mid a\in \cat L\cap \cat T^c}$ by \cite[Remark~1.26]{BarthelHeardSander23b}. Note also that for any $a\in\cat T^c$ we have $e_{\supp(a)}\in\cat L$ if and only if~$a\in \cat L$. Thus $e_Y\in\cat L$. It then remains to show $Y=\Spc(\cat T^c)$. Suppose \emph{ab absurdo} that $Y\subsetneq\Spc(\cat T^c)$. By assumption, there exist $\cat P\in Y^c$ and $b\in\cat T^c$ such that
\[
\cat P\in\supp(b)\cap Y^c\subseteq \gen(\cat P).
\]
Choose $j\in I$ such that $\cat P\in W_j=U\cap V^c$, where~$U$ and $V$ are Thomason subsets. By intersecting with $U$ we may assume that $\supp(b)$ is contained in $U$. Hence
\[
\supp(b)\cap Y^c\subseteq U\cap\gen(\cat P)\subseteq U\cap V^c.
\]
It follows that $e_{\supp(b)}\otimes f_Y\in\Loco{g_{W_j}}\subseteq \cat L$. Now the exact triangle
\[
e_{\supp(b)}\otimes e_Y\to	 e_{\supp(b)}\to e_{\supp(b)}\otimes f_Y
\]
implies $e_{\supp(b)}\in\cat L$ since the other two terms are in $\cat L$. Thus $\cat P\in\supp(b)\subseteq Y$ by the definition of $Y$, which is absurd.
\end{proof}

\begin{Rem}
\cref{thm:Hochweakscatter} strengthens \cite[Theorem~3.22]{BarthelHeardSander23b} which states that if $\Spc(\cat T^c)$ is Noetherian then $\cat T$ satisfies the local-to-global principle; see \cref{rem:hoch-scatter}. \cref{thm:Hochweakscatter} also strengthens \cite[Theorems~7.9 and 7.18]{BillySanders17pp} which prove:
\begin{enumerate}
\item If $\Spc(\cat T^c)$ is Hochster weakly scattered then $\cat T$ has the detection property.
\item If $\Spc(\cat T^c)$ is Hochster scattered and $\cat T$ admits a monoidal model then $\cat T$ satisfies the local-to-global principle.
\end{enumerate}
\end{Rem}

\begin{Rem}\label{rem:SH-ltg}
The Balmer spectrum $\Spc(\SHcp)$ satisfies the Hochster weakly scattered condition except for the $Y=\emptyset$ case (in the notation of \cref{rem:Hochweakscatter}). Nevertheless, for this example the proof of \cref{thm:Hochweakscatter} still goes through because the Thomason subset $Y$ constructed in the proof is nonempty, as shown in \cref{prop:SHp-ltg}. The stable homotopy category $\SH$ satisfies the local-to-global principle for similar reasons.
\end{Rem}

\begin{Rem}\label{rem:cosupp}
We end this section with a few words on the theory of cosupport. Dual to the tensor triangular support, a theory of tensor triangular cosupport was systematically developed (under the assumption that the Balmer spectrum is weakly Noetherian) in \cite{BCHS25}, building on prior work in \cite{HoveyStrickland99,Neeman11,BensonIyengarKrause12}. In particular, one can define the notions of costratification, colocal-to-global principle, and codetection property in terms of cosupport. Their work demonstrates that to completely understand a big tt-category one needs to consider both the support and the cosupport. Moreover, they discovered surprising relations between the theories of support and cosupport. For example, for a rigidly-compactly generated tt-category $\cat T$ with weakly Noetherian spectrum, the colocal-to-global principle, the codetection property, and the local-to-global principle are all equivalent \cite[Theorem~6.4]{BCHS25}.

We propose here a notion of cosupport which works beyond the weakly Noetherian setting. The \emph{tensor triangular cosupport} of an object $t\in \cat T$ is defined to be the set
\[
\Cosupp(t) \coloneqq \left\{ \cat P\in\Spc(\cat T^c) \left|
\begin{gathered}
[g_W,t] \neq 0 \text{ for any weakly}\\
\text{ visible subset }W\text{ containing }\cat P
\end{gathered} \right. \right\}
\]
where $[-,-]$ denotes the internal hom. This recovers the notion of cosupport in \cite[Definition~4.23]{BCHS25} when $\Spc(\cat T^c)$ is weakly Noetherian. Similarly to \cref{def:ltg}, we say that $\cat T$ satisfies the \emph{colocal-to-global principle} if we have an equality of colocalizing coideals
\[ 
\Colocid{t} = \Colocid{[g_{W_i},t] \mid i\in I}
\]
for every object $t \in \cat T$ and every cover of $\Spc(\cat T^c)$ by weakly visible subsets $W_i$. Again, this notion of colocal-to-global principle specializes to the one in \cite[Definition~4.23]{BCHS25} if $\Spc(\cat T^c)$ is weakly Noetherian. With these definitions, several results in \cite{BCHS25} still hold without the weakly Noetherian hypothesis. For example:
\end{Rem}

\begin{Thm}[Barthel--Castellana--Heard--Sanders]
The following statements are equivalent for a rigidly-compactly generated tt-category $\cat T$:
\begin{enumerate}
\item $\cat T$ satisfies the codetection property.
\item $\cat T$ satisfies the local-to-global principle.
\item $\cat T$ satisfies the colocal-to-global principle.
\end{enumerate}
\end{Thm}
\begin{proof}
The proof of (a)$\implies$(b)$\implies$(c)$\implies$(a) in \cite[Theorem~6.4]{BCHS25} carries over, \emph{mutatis mutandis}.
\end{proof}

\section{Stratification implies weakly Noetherian}\label{sec:strat}
It is natural to ask whether the tensor triangular support can be used to classify the localizing ideals of $\cat T$. More precisely:

\begin{Def}\label{def:strat}
We say that $\cat T$ is \emph{stratified} if the map
\[
\begin{split}
\big\{ \text{localizing ideals of $\cat T$}\big\} &\to \big\{\text{localizing closed subsets of $\Spc(\cat T^c)$}\big\} \\
\cat L &\mapsto \Supp(\cat L)
\end{split}
\]
is a bijection.
\end{Def}

\begin{Rem}
The map above is well-defined by \cref{prop:Supp-loc-closed}(b). If $\Spc(\cat T^c)$ is weakly Noetherian then \cref{def:strat} recovers \cite[Definition~4.4]{BarthelHeardSander23b}. In fact, we will see that if $\cat T$ is stratified in the sense of \cref{def:strat} then $\Spc(\cat T^c)$ is necessarily weakly Noetherian. Our proof is based on the comparison between the tensor triangular support and the homological support, which was studied in detail in the weakly Noetherian context in \cite{BarthelHeardSander23a}.
\end{Rem}

\begin{Rec}
Recall that each homological prime $\cat B \in \Spch(\cat T^c)$ gives rise to a pure-injective object $E_{\cat B} \in \cat T$ and the homological support of an object $t \in \cat T$ is given by
\[
\Supph(t) \coloneqq \SET{ \cat B \in \Spch(\cat T^c) }{ [t,E_{\cat B}] \neq 0}
\]
where $[-,-]$ denotes the internal hom. For every homological prime $\cat B \in \Spch(\cat T^c)$ we have $\Supph(E_{\cat B})=\{\cat B\}$. By \cite[Theorem~1.2]{Balmer20_bigsupport} the homological support satisfies the tensor product formula
\[
\Supph(t_1\otimes t_2) = \Supph(t_1) \cap \Supph(t_2) \quad \text{for any } t_1,t_2 \in \cat T.
\]
Moreover, there exists a surjective continuous map $\phi \colon \Spch(\cat T^c) \twoheadrightarrow \Spc(\cat T^c)$; see \cite[Corollary~3.9]{Balmer20_nilpotence}.
\end{Rec}

\begin{Lem}\label{lem:Supph-gW}
If $W\subseteq\Spc(\cat T^c)$ is weakly visible then $\Supph(g_W)=\phi\inv(W)$.
\end{Lem}

\begin{proof}
Apply \cite[Lemma~3.8]{BarthelHeardSander23a} and the tensor-product formula.
\end{proof}

\begin{Lem}\label{lem:Supph-inclusion}
For every $t \in \cat T$ we have $\phi(\Supph(t)) \subseteq \Supp(t)$.
\end{Lem}

\begin{proof}
If $\cat B\in\Supph(t)$ then for any weakly visible subset $W\ni\cat P\coloneqq\phi(\cat B)$ we have~$\cat B\in\phi\inv(W)=\Supph(g_W)$ by \cref{lem:Supph-gW} and hence $\cat B\in\Supph(t\otimes g_W)$ in view of the tensor-product formula. In particular, $t\otimes g_W\neq0$. This is true for every weakly visible subset $W$ containing $\cat P$, so $\cat P\in\Supp(t)$.
\end{proof}

\begin{Lem}\label{lem:Supp-EB}
For every $\cat B \in \Spch(\cat T^c)$ we have $\Supp(E_{\cat B})=\{\phi(\cat B)\}$.
\end{Lem}

\begin{proof}
Let $\cat P\coloneqq\phi(\cat B)$. First we show that $E_{\cat B}$ is $\cat P$-local, \ie $E_{\cat B}\simeq E_{\cat B}\otimes f_{Y_{\cat P}}$. For any object $x\in \cat P$ we have $\cat P\notin\supp(x)$. By \cref{lem:Supph-inclusion} we have $\cat B\notin\Supph(x)$, that is, $[x,E_{\cat B}]=0$. This holds for every $x\in\cat P$, so $E_{\cat B}\in\Loc{\cat P}^{\perp}=f_{Y_{\cat P}}\otimes\cat T$ is $\cat P$-local. Thus $\Supp(E_{\cat B})\subseteq \gen(\cat P)$. On the other hand, we claim that $E_{\cat B}\in\Loco{e_{\supp(a)}}$ for any object $a\notin\cat P$. Indeed, we have
\begin{align*}
a\notin\cat P & \iff \cat P\in \supp(a) \\
& \iff \cat B\notin \phi\inv(\supp(a)^c)= \Supph(f_{\supp(a)}) && \text{by \cref{lem:Supph-gW}} \\
& \iff E_{\cat B} \otimes f_{\supp(a)}=0 && \\
& \iff E_{\cat B} \in \Loco{e_{\supp(a)}} && \text{by \eqref{eq:kerfY}}
\end{align*}
where the second-to-last equivalence is due to \cite[Theorem~1.8]{Balmer20_bigsupport} and the fact that $f_{\supp(a)}$ is a (weak) ring object. Therefore we have $E_{\cat B}\otimes e_{\supp(a)}\simeq E_{\cat B}$ for any $a\notin\cat P$. Invoking the half-tensor formula (\cref{lem:half-tensor}) we obtain
\[
\Supp(E_{\cat B})\subseteq \bigcap_{a\notin\cat P}\supp(a)=\overline{\{\cat P\}}.
\]
Therefore
\[
\Supp(E_{\cat B})\subseteq \overline{\{\cat P\}}\cap\gen(\cat P)=\{\cat P\}.
\]
It remains to prove $\cat P\in\Supp(E_{\cat B})$. From what we have shown it follows that $0\neq E_{\cat B}\simeq E_{\cat B}\otimes f_{Y_{\cat P}}\simeq E_{\cat B}\otimes e_{\supp(a)} \otimes f_{Y_{\cat P}}	$ for any $a\notin \cat P$. By \cref{rem:Supp-def} we conclude that $\cat P\in\Supp(E_{\cat B})$.
\end{proof}

\begin{Exa}
By \cite[Corollary~3.6]{BalmerCameron21}, the Morava K-theory spectrum $K(p,n)$, for a prime $p$ and $0\le n\le \infty$, is isomorphic to $E_{\cat B_{p,n}}$, where $\cat B_{p,n}\in\Spch(\SHc)$ is the homological prime corresponding to the Balmer prime $\cat P_{p,n}\in\Spc(\SHc)$. We then have $\Supp(K(p,n))=\{\cat P_{p,n}\}$ by \cref{lem:Supp-EB}. In particular, the mod-$p$ Eilenberg-Maclane spectrum $\HFp=K(p,\infty)$ is supported on the singleton $\{\cat P_{p,\infty}\}$.
\end{Exa}

\begin{Cor}\label{cor:surjectivity}
If $\cat T$ satisfies the detection property then the map in \cref{def:strat} is surjective.
\end{Cor}

\begin{proof}
Let $F\subseteq\Spc(\cat T^c)$ be a localizing closed subset. Choosing a homological prime $\cat B_{\cat P}\in\phi\inv(\{\cat P\})$ for every $\cat P\in F$, by \cref{lem:Supp-EB} and \cite[Theorem~4.7(4)]{BillySanders17pp} we have
\[
\Supp(\Loco{E_{\cat B_{\cat P}}\mid \cat P\in F})=\overline{\bigcup_{\cat P\in F}\Supp(E_{\cat B_{\cat P}})}^{\loc}=\overline{F}^{\loc}=F. \qedhere
\]
\end{proof}

\begin{Rem}
In \cite[Lemma~3.4]{BarthelHeardSander23b} it was proved that if $\Spc(\cat T^c)$ is weakly Noetherian then the map above is surjective (without assuming the detection property).
\end{Rem}

\begin{Exa}
The Balmer spectrum $\Spc(\SHcp)$ is not weakly Noetherian (\cref{exa:SH-non-weakly-noetherian}) but $\SHp$ satisfies the detection property (\cref{prop:SHp-detection}). Thus every localizing closed subset of $\Spc(\SHcp)$ can be realized as the support of some localizing ideal of $\SHcp$. Moreover, a subset $S\subseteq\Spc(\SHcp)$ is localizing closed if and only if either $\cat P_{p,\infty}\in S$ or $\cat P_{p,\infty}\notin S$ and $S$ is finite. This follows from that fact that the Thomason subsets of $\Spc(\SHcp)$ are of the form $\overline{\{\cat P_{p,i}\}}$ for $0\le i < \infty$; see \cref{exa:SH-non-weakly-noetherian}.
\end{Exa}

\begin{Rem}
The following result indicates that the weakly Noetherian hypothesis in \cite[Theorem~4.7]{BarthelHeardSander23a} is unnecessary.
\end{Rem}

\begin{Prop}\label{prop:stratified-supph}
If $\cat T$ is stratified then $\phi(\Supph(t))=\Supp(t)$ for any $t\in\cat T$.
\end{Prop}

\begin{proof}
By \cref{lem:Supph-inclusion} the inclusion $\phi(\Supph(t)) \subseteq \Supp(t)$ always holds. To prove the other inclusion, let $\cat P\in\Supp(t)$ and choose any $\cat B\in\phi\inv(\{\cat P\})$. In light of \cref{lem:Supp-EB}, we have $\Supp(E_{\cat B})=\{\cat P\}\subseteq\Supp(t)$ and thus $E_{\cat B}\in\Loco{t}$ due to stratification, which implies $[t,E_{\cat B}]\neq0$ since $[E_{\cat B},E_{\cat B}]\neq0$. This completes the proof.
\end{proof}

\begin{Thm}\label{thm:stratified-weakly-noetherian}
If $\cat T$ is stratified then $\Spc(\cat T^c)$ is weakly Noetherian.
\end{Thm}

\begin{proof}
Let $W$ be any subset of $\Spc(\cat T^c)$. It suffices to show that $W$ is localizing closed by \cref{rem:loc-top}. To see this, we choose a $\cat B_{\cat P}\in\phi\inv(\{\cat P\})$ for every $\cat P\in W$ and consider the object $t_W\coloneqq \coprod_{\cat P\in W}E_{\cat B_{\cat P}}$. By \cite[Proposition~4.3(b)]{Balmer20_bigsupport} we have $\Supph(t_W)=\SET{\cat B_{\cat P}}{\cat P \in W}$. It then follows from \cref{prop:stratified-supph} and \cref{prop:Supp-loc-closed}(a) that $W=\Supp(t_W)$ is localizing closed.
\end{proof}

\begin{Rem}
The theorem above shows that the generalization of the Balmer--Favi support to the tensor triangular support does not broaden the scope of the stratification theory developed in \cite{BarthelHeardSander23b}.
\end{Rem}

\begin{Exa}\label{exa:SH-non-stratified}
The Balmer spectrum $\Spc(\SH^c)$ is not weakly Noetherian and therefore the stable homotopy category $\SH$ is not stratified; \cf \cref{rem:SH-ltg}.
\end{Exa}

\section{Comparison of support theories}\label{sec:comparison}
Our final goal is to study the relation between the canonical BIK support and the tensor triangular support for a rigidly-compactly generated tensor triangulated category $\cat T$, via the comparison map $\rho:\Spc(\cat T^c)\to\Spech(\End^*_{\cat T}(\unit))$ introduced in \cite{Balmer10a}.

\begin{Rem}\label{rem:BIK-finite}
Recall from \cref{exa:graded-algebraic-local} that for a prime ideal $\frakp\in\Spech(\End^*_{\cat T}(\unit))$ the BIK localization functor $L_{\cal Z(\frakp)}$ is the finite localization functor associated to the Thomason subset $\rho\inv(\cal Z(\frakp))$. Let $\fraka=(x_1,\ldots,x_n)$ be a finitely generated homogeneous ideal of $\End^*_{\cat T}(\unit)$ with homogeneous generators $\{x_i\}_{1\le i\le n}$. We then have
\[
\cat T_{\cal V(\fraka)}=\Loc{\kos x{\fraka}\mid x\in\cat T^c}=\Loco{\kos \unit{\fraka}}=\Loco{\bigotimes_{1\le i\le n}\cone(x_i)}
\]
where the first equality is due to \cref{prop:bikalgebraic}(a) and the second equality is explained in the proof of \cref{prop:suppbiksmash}. It follows that the BIK colocalization functor $\varGamma_{\cal V(\fraka)}$ is the finite colocalization functor associated to the Thomason subset
\[
\bigcap_{1\le i\le n}\supp(\cone(x_i))=\bigcap_{1\le i\le n}\rho\inv(\cal V(x_i))=\rho\inv(\cal V(\fraka)).
\]
Therefore
\begin{equation}\label{eq:BIK-finite}
\varGamma_{\cal V(\fraka)}L_{\cal Z(\frakp)}t\simeq e_{\rho\inv(\cal V(\fraka))}\otimes f_{\rho\inv(\cal Z(\frakp))}\otimes t
\end{equation}
for every $t\in \cat T$.
\end{Rem}

\begin{Thm}\label{thm:bikcomparison}
Let $\cat T$ be a rigidly-compactly generated tensor triangulated category. Consider the comparison map $\rho:\Spc(\cat T^c)\to\Spech(\End^*_{\cat T}(\unit))$. Then:
\begin{enumerate}
\item $\rho(\Supp(t))\subseteq\SuppBIK(t)$ for every $t\in\cat T$.
\item If $\rho$ is a homeomorphism then $\rho(\Supp(t))=\SuppBIK(t)$ for every $t\in\cat T$.
\end{enumerate}
\end{Thm}

\begin{proof}
Let $\cat P\in\Supp(t)$. If $\frakp\coloneqq\rho(\cat P)\notin\SuppBIK(t)$ then by \eqref{eq:BIK-finite} there exists a finitely generated homogeneous ideal $\fraka\subseteq\frakp$ such that
\[
0=\varGamma_{\cal V(\fraka)}L_{\cal Z(\frakp)}t=e_{\rho\inv(\cal V(\fraka))}\otimes f_{\rho\inv(\cal Z(\frakp))}\otimes t.
\]
Since the point $\cat P$ is contained in the weakly visible subset $\rho\inv(\cal V(\fraka))\cap\rho\inv(\cal Z(\frakp))^c$, we have $\cat P\notin\Supp(t)$, which establishes (a). To show (b), let $\frakp\in\Spech(\End^*_{\cat T}(\unit))$ correspond to some $\cat P\in\Spc(\cat T^c)$. Note that $\rho\inv(\cal Z(\frakp))=Y_{\cat P}$ because $\rho$ is a homeomorphism. Then observe that
\begin{align*}
\frakp\notin\SuppBIK(t) & \iff \exists \text{ a Thomason closed }\cal V\ni\frakp: e_{\rho\inv(\cal V)}\otimes f_{\rho\inv(\cal Z(\frakp))}\otimes t=0 \\
& \iff \exists \text{ a Thomason closed }V\ni\cat P: e_{V}\otimes f_{Y_{\cat P}}\otimes t=0 \\
& \iff \cat P\notin\Supp(t). \qedhere
\end{align*}
\end{proof}

\begin{Cor}
For any compact object $x\in\cat T^c$ we have
\[
\rho(\supp(x))\subseteq\SuppBIK(x)\subseteq\Supp_{\End^*_{\cat T}(\unit)} \End^*_{\cat T}(x).
\]
Moreover, if $\rho$ is a homeomorphism then these inclusions are equalities. 
\end{Cor}

\begin{proof}
The first inclusion is a special case of \cref{thm:bikcomparison}(a). For the second inclusion, let $\frakp\notin\Supp_{\End^*_{\cat T}(\unit)} \End^*_{\cat T}(x)$. By \cite[(2.2)]{Lau23} we have $L_{\cal Z(\frakp)}x=0$ and thus $\frakp\notin\SuppBIK(x)$. This establishes the second inclusion. The equalities follow from \cite[Proposition~2.10]{Lau23}.
\end{proof}

\begin{Rem}
We now give a notion of stratification with respect to the canonical BIK support function which takes values in $\SuppBIK(\cat T) \subseteq \Spech(\End_{\cat T}^*(\unit))$:
\end{Rem}

\begin{Def}\label{def:coh-strat}
We say that $\cat T$ is \emph{cohomologically stratified} if the map
\[
\begin{split}
\big\{ \text{localizing ideals of $\cat T$}\big\} &\to \big\{\text{closed subsets of $\SuppBIK(\cat T)$}\big\} \\
\cat L &\mapsto \SuppBIK(\cat L)
\end{split}
\]
is a bijection, where $\SuppBIK(\cat T)$ is equipped with the subspace topology of the localizing topology on $\Spech(\End_{\cat T}^*(\unit))$.
\end{Def}

\begin{Rem}\label{rem:coh-strat}
The map above is well-defined by \cref{prop:suppbikloc}(b). Moreover, if the comparison map $\rho\colon\Spc(\cat T^c)\to\Spech(\End_{\cat T}^*(\unit))$ is surjective then we have $\SuppBIK(\cat T)=\SuppBIK(\unit)=\Spech(\End_{\cat T}^*(\unit))$ by part (a) of \cref{thm:bikcomparison}.
\end{Rem}

\begin{Exa}\label{exa:rho-surjective}
If $\End_{\cat T}^*(\unit)$ is Noetherian then $\rho$ is surjective by \cite[Theorem~7.3]{Balmer10a} and hence $\SuppBIK(\cat T)=\Spech(\End_{\cat T}^*(\unit))$.
\end{Exa}

\begin{Exa}\label{exa:comparison-DA}
For a commutative ring $A$, the unbounded derived category $\Der(A)$ is rigidly-compactly generated and the derived category $\Derperf(A)$ of perfect complexes is its subcategory of rigid-compact objects. The associated comparison map $\Spc(\Derperf(A)) \to \Spec(A)$ is a homeomorphism by \cite[Proposition~8.1]{Balmer10a}, so we have $\SuppBIK(\Der(A))=\Spec(A)$.
\end{Exa}

\begin{Cor}\label{cor:coh-strat}
If the comparison map $\rho$ is a homeomorphism then $\cat T$ is cohomologically stratified if and only if it is stratified.
\end{Cor}

\begin{proof}
By \cref{rem:coh-strat} the BIK space of supports $\SuppBIK(\cat T)$ is $\Spech(\End^*_{\cat T}(\unit))$. The result thus follows from \cref{thm:bikcomparison}(b).
\end{proof}

\begin{Exa}
Let $A$ be a commutative ring. By \cref{exa:comparison-DA} we can identify $\Spc(\Derperf(A))$ with $\Spec(A)$ via the comparison map, under which the tensor triangular support and the canonical BIK support coincide, according to \cref{thm:bikcomparison}. For any prime $\frakp=\rho(\cat P)\in\Spec(A)$ and any finitely generated ideal $\fraka\subseteq\frakp$, it follows from \cite[Lemma~5.1]{BillySanders17pp} that
\[
f_{\rho\inv(\cal Z(\frakp))} = f_{Y_{\cat P}} \simeq A_{\frakp} \text{ and } e_{\rho\inv(\cal V(\fraka))} = e_{\supp(\kos \unit{\fraka})}\simeq K^{\infty}(\fraka)
\]
where $K^{\infty}(\fraka)$ is the stable Koszul complex of $\fraka$. Hence for a complex $X\in\Der(A)$ we have
\[
\Supp(X)= \left\{ \frakp\in\Spec(A) \left| 
{\begin{gathered}
K^{\infty}(\fraka)\otimes X_{\frakp} \neq 0 \text{ for any finitely} \\
\text{generated ideal $\fraka$ contained in $\frakp$}
\end{gathered}} \right. \right\}.
\]
This notion of support for complexes over a (non-Noetherian) commutative ring was first proposed and studied in \cite{BillySanders17pp}. The condition $K^{\infty}(\fraka)\otimes X_{\frakp} \neq 0$ is equivalent to $A/\fraka\otimes X_{\frakp} \neq 0$ since $\Loco{K^{\infty}(\fraka)}=\Loco{A/\fraka}$ by \cite[Proposition~5.6]{Greenlees01}. If the ideal $\frakp$ itself is finitely generated then this condition amounts to $\kappa(\frakp)\otimes X \neq 0$. Therefore, when $A$ is Noetherian the tensor triangular support recovers the support defined in \cite{Foxby79}.
\end{Exa}

\begin{Rem}\label{rem:Neeman-D(A)}
Neeman proved that $\Der(A)$ is (cohomologically) stratified whenever $A$ is Noetherian; see \cite[Theorem~2.8]{Neeman92a}. This result was extended to the absolutely flat approximations of topologically Noetherian commutative rings by Stevenson; see \cite[Theorem~4.23]{Stevenson14a}. On the other hand, Neeman \cite{Neeman00} gave an example of a non-Noetherian commutative ring such that the stratification fails. It remains an open question to determine for which commutative rings stratification holds. However, our \cref{cor:coh-strat} and \cref{thm:stratified-weakly-noetherian} show that for any commutative ring $A$, if $\Der(A)$ is stratified then $\Spec(A)$ is necessarily weakly Noetherian.
\end{Rem}

\begin{Rem}
In \cite[Theorem~5.5(3)]{BillySanders17pp} it was shown that if the prime ideals of a commutative ring $A$ satisfy the descending chain condition then $\Der(A)$ has the detection property. This can also be deduced from \cref{thm:bikcomparison} and \cref{cor:dccbik}. In fact, our \cref{cor:dccbik} generalizes \cite[Theorem~5.5(2)]{BillySanders17pp}.
\end{Rem}

\begin{Exa}
For a Noetherian scheme $X$, the derived category $\Derqc(X)$ of complexes of $\cat O_X$-modules with quasi-coherent cohomology is stratified; see \cite[Corollary~5.10]{BarthelHeardSander23b}. However, $\Derqc(X)$ is not cohomologically stratified in general, since the graded endomorphism ring of the unit object in this category, that is, the sheaf cohomology ring $H(X,\cat O_X)$, may not have enough prime ideals when $X$ is nonaffine; see \cite[Remark~8.2]{Balmer10a}, for example.
\end{Exa}

\begin{Exa}\label{exa:stmod}
Let $G$ be a finite group and $k$ a field of characteristic $p>0$ such that $p$ divides the order of $G$. The big stable module category $\StMod(kG)$ is BIK-stratified by $H^*(G,k)$ (\cite[Theorem~10.3]{BensonIyengarKrause11a}). The associated Balmer spectrum $\Spc(\stmod(kG))$ is homeomorphic to $\Proj(H^*(G,k))$, which is a Noetherian space since $H^*(G,k)$ is a Noetherian ring by the Evens-Venkov theorem \cite[II(3.10)]{Benson98}. Note that this BIK-stratification for $\StMod(kG)$ is not canonical since the graded endomorphism ring of the unit for $\StMod(kG)$ is not $H^*(G,k)$ but rather the Tate cohomology ring $\hat{H}^*(G,k)$; see \cite[page~26]{BensonKrause02}. The original statement of Benson--Iyengar--Krause cannot be applied to the canonical action of $\hat{H}^*(G,k)$ on $\StMod(kG)$ since $\hat{H}^*(G,k)$ is rarely Noetherian. In fact, $\hat{H}^*(G,k)$ is Noetherian if and only if the $p$-rank of $G$ is $1$ if and only if $\hat{H}^*(G,k)$ is periodic; see \cite[Lemma~10.1]{BensonIyengarKrause08}.

On the other hand, $\StMod(kG)$ is stratified in the sense of \cref{def:strat} by \cite[Example~7.12]{BarthelHeardSander23b}. In the following we will show that $\StMod(kG)$ is also cohomologically stratified in the sense of \cref{def:coh-strat}. In other words, it is canonically stratified by $\hat{H}^*(G,k)$. Note that this does not follow directly from \cref{thm:bikcomparison}(b) because in this example the comparison map $\rho$ is not a homeomorphism in general, as we shall see below.
\end{Exa}

\begin{Thm}\label{thm:stmod}
Let $G$ be a finite group and $k$ a field of characteristic $p>0$ such that $p$ divides the order of $G$. The stable module category $\StMod(kG)$ is cohomologically stratified.
\end{Thm}

\begin{proof}
By Rickard \cite{Rickard89} the small stable module category $\stmod(kG)$ is equivalent to the quotient $\Derb(\mmod{kG})/\Derperf(kG)$. Note that the graded endomorphism ring of the unit object of $\cat K\coloneqq\Derb(\mmod{kG})$ is isomorphic to the group cohomology ring. Therefore, we can identify $H^*(G,k)$ with $\End_{\cat K}^*(\unit)$ and $\hat{H}^*(G,k)$ with $\End_{\stmod(kG)}^*(\unit)$; see \cref{exa:stmod}.

Now consider the functor $q\colon \cat K \to \cat K/\Derperf(kG) \simeq \stmod(kG)$. The naturality of the comparision map gives us the following commutative diagram:
\begin{equation}\label{eq:diagram1}
\begin{tikzcd}
\Spc(\stmod(kG)) \ar[d, hook, "\rho"] \ar[r, hook,"\Spc(q)"] & \Spc(\cat K) \ar[d,"\rho_{\cat K}","\simeq"'] \\
\Spech(\hat{H}^*(G,k)) \ar[r,"\Spech(\iota)"] &\Spech(H^*(G,k))
\end{tikzcd}
\end{equation}
where $\Spc(q)$ is an open embedding and $\rho_{\cat K}$ is a homeomorphism by \cite[Proposition~8.5]{Balmer10a}. Moreover, $\iota\colon H^*(G,k) \hook \hat{H}^*(G,k)$ is the first map that appears in \cite[(10.2)]{BensonIyengarKrause11a}, which views $H^*(G,k)$ as a subring of $\hat{H}^*(G,k)$; see also \cite[(2.1)]{BensonKrause02}. It follows that $\rho$ is injective. On the other hand, we have the following commutative diagram from the proof of \cite[Proposition~8.5]{Balmer10a}:
\begin{equation}\label{eq:diagram2}
\begin{tikzcd}
\Spc(\stmod(kG)) \ar[d, "\varphi\inv","\simeq"'] \ar[r,hook,"\Spc(q)"] & \Spc(\cat K) \ar[d,"\rho_{\cat K}","\simeq"'] \\
\Proj(H^*(G,k)) \ar[r,hook] &\Spech(H^*(G,k))
\end{tikzcd}
\end{equation}
in which $\varphi \colon \Proj(H^*(G,k)) \xra{\sim} \Spc(\stmod(kG))$ is the homeomorphism described in \cite[Corollary~5.10]{Balmer05a} and $\Proj(H^*(G,k)) \hook \Spech(H^*(G,k))$ is the canonical open embedding which misses the unique closed point $H^+(G,k)$ in $\Spech(H^*(G,k))$. Combining \eqref{eq:diagram1} and \eqref{eq:diagram2} we obtain a commutative diagram:
\begin{equation}\label{eq:diagram3}
\begin{tikzcd}
\Spc(\stmod(kG)) \ar[d, hook, "\rho"] \ar[r,"\varphi\inv","\simeq"'] & \Proj(H^*(G,k)) \ar[d,hook] \\
\Spech(\hat{H}^*(G,k)) \ar[r,"\Spech(\iota)"] &\Spech(H^*(G,k)).
\end{tikzcd}
\end{equation}

If $\hat{H}^*(G,k)$ is Noetherian then \cite[Theorem~7.3]{Balmer10a} implies that $\rho$ is surjective and hence a bijection. By \cite[Lemma~10.1]{BensonIyengarKrause08} $\hat{H}^*(G,k)$ being Noetherian is equivalent to that the $p$-rank of $G$ equals $1$, which implies that the Krull dimension of $\Spech(H^*(G,k))$ is equal to $1$ by Quillen stratification theorem \cite{Quillen71}. It follows that $\Spc(\stmod(kG))\cong \Proj(H^*(G,k))$ has zero Krull dimension, that is, $\Spc(\stmod(kG))$ is a discrete space. Moreover, since the trivial representation is indecomposable, $\Spc(\stmod(kG))$ is a singleton by \cite[Theorem~2.11]{Balmer07}, which forces $\rho$ to be a homeomorphism. Therefore, $\StMod(kG)$ is cohomologically stratified by \cref{cor:coh-strat}.

If $\hat{H}^*(G,k)$ is not Noetherian (\ie the $p$-rank of $G$ is at least $2$) then the negative part $\hat{H}^-(G,k)$ is nilpotent by \cite[Proposition~2.4]{BensonKrause02}. It follows that
\[
\begin{aligned}
\Spech(\iota) \colon \Spech(\hat{H}^*(G,k)) & \to \Spech(H^*(G,k)) \\
\frakp=\hat{H}^-(G,k)\oplus \frakp^{\ge0} & \mapsto \frakp^{\ge0}
\end{aligned}
\]
is a homeomorphism where $\frakp^{\ge0}$ denotes the nonnegative part of a graded prime $\frakp$. On the other hand, by \eqref{eq:diagram3} the map $\rho \colon \Spc(\stmod(kG)) \hook \Spech(\hat{H}^*(G,k))$ is an open embedding which misses the unique closed point $\mathfrak{m}\coloneqq\hat{H}^{i\neq0}(G,k)$. Since $\Spech(\hat{H}^*(G,k)) \cong \Spech(H^*(G,k))$ is Noetherian, $\cal V(\mathfrak{m})=\{\mathfrak{m}\}$ is Thomason closed. We thus have $e_{\rho\inv(\cal V(\mathfrak{m}))}\otimes f_{\rho\inv(\cal Z(\mathfrak{m}))} = e_{\emptyset}\otimes f_{\Spc(\stmod(kG))} \simeq 0\otimes \unit = 0$ and hence $\mathfrak{m}\notin\SuppBIK(t)$ for every $t\in \StMod(kG)$. It then follows from \cref{thm:bikcomparison}(a) that
\begin{equation}\label{eq:imrho}
\SuppBIK(\cat T)=\SuppBIK(\unit)=\im\rho.
\end{equation}
Now suppose that $\frakp=\rho(\cat P)\in\im\rho$ is a nonclosed point in $\Spech(\hat{H}^*(G,k))$. Note that $\{\frakp\}=\cal V(\frakp)\cap \cal Z(\frakp)^c$, where $\cal V(\frakp)$ is Thomason closed since $\Spech(\hat{H}^*(G,k))$ is Noetherian. It follows that for every $t\in \StMod(kG)$ we have
\[
\frakp\in\SuppBIK(t) \iff e_{\rho\inv(\cal V(\frakp))} \otimes f_{\rho\inv(\cal Z(\frakp))} \otimes t \neq 0 \iff \cat P \in \Supp(t).
\]
Therefore $\rho(\Supp(t))=\SuppBIK(t)$ for all $t\in\StMod(kG)$. From \eqref{eq:imrho} and the fact that $\StMod(kG)$ is stratified, we conclude that $\StMod(kG)$ is cohomologically stratified.
\end{proof}

\begin{Rem}
Our definition of cohomological stratification (\cref{def:coh-strat}) generalizes the one given by \cite[Definition~2.21]{BCHNP25} which requires $\cat T$ to be Noetherian (\cite[Definition~2.9]{BCHNP25}) and the comparison map $\rho$ to be a homeomorphism. Indeed, if $\cat T$ is Noetherian then $\End_{\cat T}^*(\unit)$ is Noetherian and thus every subset of $\SuppBIK(\cat T)$ is localizing closed. Moreover, if $\rho$ is a homeomorphism then $\SuppBIK(\cat T)=\Spech(\End_{\cat T}^*(\unit))$ by \cref{rem:coh-strat}. Note, however, that our \cref{def:coh-strat} does not put any restriction on $\rho$. As \cref{thm:stmod} shows, requiring~$\rho$ to be a homeomorphism would eliminate interesting examples of cohomologically stratified categories in the non-Noetherian context.
\end{Rem}

\raggedbottom 

\bibliographystyle{alpha}\bibliography{supp}

\end{document}